\documentclass[12pt]{article}

\usepackage{fontspec, xunicode, xltxtra}  
\usepackage{mathtools}
\usepackage{amsmath,diagbox}
\usepackage{indentfirst}
\usepackage{mathrsfs}
\usepackage{amsfonts}
\usepackage{arydshln}
\usepackage{enumerate}
\usepackage{setspace}
\usepackage{amssymb,amsthm,cases}
\usepackage{amsbsy,pifont}
\usepackage{latexsym,diagbox,tabularx}
\usepackage{amsmath,amscd,bbm,multirow}
\usepackage{graphicx,epsfig,extarrows,dsfont,mathtools}
\usepackage[final,allcolors=blue,colorlinks=true]{hyperref}
\usepackage[normalem]{ulem}
\input xy
\usepackage{caption,subcaption} 
\usepackage{tikz}
\usepackage{tikz-cd}
\usetikzlibrary{decorations.pathreplacing,decorations.markings}
\usetikzlibrary{cd}
\usetikzlibrary{shapes.geometric}
\usetikzlibrary{arrows,arrows.meta}
\usetikzlibrary{graphs,positioning}
\usetikzlibrary{matrix,decorations.pathmorphing}
\usetikzlibrary{math,calc,intersections,through,angles,arrows.meta,shapes.geometric,shadows,quotes,spy,datavisualization,datavisualization.formats.functions,plotmarks}
\tikzset{every picture/.style={samples=300,smooth,line join=round,thick,>=stealth}}
\usepackage{enumerate}
\usepackage{fancyhdr}
\usepackage[sectionbib]{chapterbib}

\tikzset{
	on each segment/.style={
		decorate,
		decoration={
			show path construction,
			moveto code={},
			lineto code={
				\path [#1]
				(\tikzinputsegmentfirst) -- (\tikzinputsegmentlast);
			},
			curveto code={
				\path [#1] (\tikzinputsegmentfirst)
				.. controls
				(\tikzinputsegmentsupporta) and (\tikzinputsegmentsupportb)
				..
				(\tikzinputsegmentlast);
			},
			closepath code={
				\path [#1]
				(\tikzinputsegmentfirst) -- (\tikzinputsegmentlast);
			},
		},
	},
	mid arrow/.style={postaction={decorate,decoration={
				markings,
				mark=at position .5 with {\arrow[#1]{stealth}}
	}}},
}

\numberwithin{equation}{section}

\theoremstyle{plain}
\newtheorem{theorem}{Theorem}[section]

\newtheorem{corollary}[theorem]{Corollary}

\newtheorem{lemma}[theorem]{Lemma}

\theoremstyle{definition}
\newtheorem{definition}[theorem]{Definition}

\newtheorem{remark}[theorem]{Remark}

\def\XXint#1#2#3{{\setbox0=\hbox{$#1{#2#3}{\int}$}
		\vcenter{\hbox{$#2#3$}}\kern-.5\wd0}}

\DeclareMathSymbol{\subseteq}{\mathrel}{symbols}{"12}
\DeclareMathSymbol{\supseteq}{\mathrel}{symbols}{"13} 
\DeclareMathSymbol{\subsetneq}{\mathrel}{AMSb}{"28}                  \DeclareMathSymbol{\supsetneq}{\mathrel}{AMSb}{"29}    

\DeclareMathSymbol{\nsubseteq}{\mathrel}{AMSb}{"2A}                  

\DeclareMathSymbol{\nsupseteq}{\mathrel}{AMSb}{"2B}

\DeclareMathDelimiter{\langle}{\mathop}{symbols}{"68}{largesymbols}{"0A}
\DeclareMathDelimiter{\rangle}{\mathclose}{symbols}{"69}{largesymbols}{"0B}

\DeclareSymbolFont{txfontsA}{U}{txmia}{m}{it}
\SetSymbolFont{txfontsA}{bold}{U}{txmia}{bx}{it}
\DeclareFontSubstitution{U}{txmia}{m}{it}
\DeclareMathSymbol{\upalpha}{\mathord}{txfontsA}{"0B}
\DeclareMathSymbol{\upbeta}{\mathord}{txfontsA}{"0C}
\DeclareMathSymbol{\upgamma}{\mathord}{txfontsA}{"0D}
\DeclareMathSymbol{\updelta}{\mathord}{txfontsA}{"0E}
\DeclareMathSymbol{\upepsilon}{\mathord}{txfontsA}{"0F}
\DeclareMathSymbol{\upzeta}{\mathord}{txfontsA}{"10}
\DeclareMathSymbol{\upeta}{\mathord}{txfontsA}{"11}
\DeclareMathSymbol{\uptheta}{\mathord}{txfontsA}{"12}
\DeclareMathSymbol{\upiota}{\mathord}{txfontsA}{"13}
\DeclareMathSymbol{\upkappa}{\mathord}{txfontsA}{"14}
\DeclareMathSymbol{\uplambda}{\mathord}{txfontsA}{"15}
\DeclareMathSymbol{\upmu}{\mathord}{txfontsA}{"16}
\DeclareMathSymbol{\upnu}{\mathord}{txfontsA}{"17}
\DeclareMathSymbol{\upxi}{\mathord}{txfontsA}{"18}
\DeclareMathSymbol{\uppi}{\mathord}{txfontsA}{"19}
\DeclareMathSymbol{\uprho}{\mathord}{txfontsA}{"1A}
\DeclareMathSymbol{\upsigma}{\mathord}{txfontsA}{"1B}
\DeclareMathSymbol{\uptau}{\mathord}{txfontsA}{"1C}
\DeclareMathSymbol{\upupsilon}{\mathord}{txfontsA}{"1D}
\DeclareMathSymbol{\upphi}{\mathord}{txfontsA}{"1E}
\DeclareMathSymbol{\upchi}{\mathord}{txfontsA}{"1F}
\DeclareMathSymbol{\uppsi}{\mathord}{txfontsA}{"20}
\DeclareMathSymbol{\upomega}{\mathord}{txfontsA}{"21}
\DeclareMathSymbol{\upvarepsilon}{\mathord}{txfontsA}{"22}
\DeclareMathSymbol{\upvartheta}{\mathord}{txfontsA}{"23}
\DeclareMathSymbol{\upvarpi}{\mathord}{txfontsA}{"24}
\DeclareMathSymbol{\upvarrho}{\mathord}{txfontsA}{"25}
\DeclareMathSymbol{\upvarsigma}{\mathord}{txfontsA}{"26}
\DeclareMathSymbol{\upvarphi}{\mathord}{txfontsA}{"27}                   

\DeclareSymbolFont{ugmL}{OMX}{mdugm}{m}{n}
\SetSymbolFont{ugmL}{bold}{OMX}{mdugm}{b}{n}
\DeclareMathAccent{\wideparen}{\mathord}{ugmL}{"F3}

\def\pt{\partial}

\def\ra{\rightarrow}

\def\s{\subseteq}

\def\e{\epsilon}

\def\ol{\overline}

\def\vp{\varphi}
\def\lg{\langle}

\def\rg{\rangle}

\def\bf{\textbf}
\def\pt{\partial}

\def\Om{\Omega}
\def\la{\lambda}
\def\al{\alpha}

\def\de{\delta}
\def\ga{\gamma}

\def\Ga{\Gamma}
\def\La{\Lambda}

\def\ts{\times}

\def\iy{\infty}

\def\ots{\otimes}
\def\f{\frac}

\def\Lra{\Leftrightarrow}

\def\df{\mathrm d}
\def\im{\mathrm i}

\def\wt{\widetilde}
\def\wh{\widehat}

\def\esssup{\operatorname*{ess\ \! sup}}

\def\hra{\hookrightarrow}

\def\dra{\downarrow}
\def\mcL{\mathcal{L}}

\def\mcA{\mathcal{A}}

	\DeclareMathOperator{\Div}{div}

	\DeclareMathOperator{\Real}{Re}
	\DeclareMathOperator{\Image}{Im}

	\DeclareMathOperator{\dist}{dist}

	\DeclareMathOperator{\supp}{{supp}}

	\newcommand{\R}{\mathbb R}
	\newcommand{\N}{\mathbb N}

	\newcommand{\Cc}{\mathbb C}

\begin{document} 
	
	
	\title{Observability Inequalities for Degenerate Schr\"odinger Equations via a Shape Approximation Method}
	\author{Dong-Hui Yang$^a$, Hongli Sun$^b$, Weijia Wu$^c$, Yuanzhi Zhou$^a$\footnote{\small The corresponding author:  262111122@csu.edu.cn }}

	\maketitle{}
	\thispagestyle{empty}
	\thispagestyle{empty}
	
	\begin{abstract}
		   We study a two-dimensional Schrödinger equation with a boundary degeneracy. More precisely, the diffusion matrix is given by $
		   A=\operatorname{diag}(1,x_2^\alpha), \alpha\in(0,1)$, 
		   so that the partial differential operator degenerates on a portion of the boundary. The main objective is to establish a boundary observability inequality for the corresponding degenerate Schrödinger equation.
		   
		   To overcome the difficulties caused by the boundary degeneracy, we introduce a shape design and approximation procedure. By removing a thin neighborhood of the degenerate boundary, we construct a family of uniformly   Schrödinger equations on perturbed domains. We first investigate the spectral properties of the associated operators and establish the convergence of the approximate problems to the original degenerate problem. We then derive a boundary observability inequality for the uniformly  approximating equations and pass to the limit as the thickness of the removed boundary layer tends to zero. This yields a boundary observability inequality for the original degenerate Schrödinger equation.
		   
		   Our approach combines spectral analysis, separation of variables, and shape design techniques. By employing multiplier methods for the uniformly   approximating problems and passing to the degenerate limit, we establish the boundary observability inequality for the original degenerate Schrödinger equation.

		   \vspace{0.3cm}
		   
		   
		   \vspace{0.3cm}
		   
	\end{abstract}
	
	\section{Introduction}

	Controllability theory originated in the 1950s and has since developed into a broad and active research area in mathematics, as well as in science and engineering. In the context of partial differential equations (PDEs), controllability has attracted considerable attention due to its theoretical importance and numerous applications. In particular, controllability is closely related to unique continuation and other fundamental properties of PDEs.
	
	There is an extensive literature on the controllability of uniformly parabolic, hyperbolic, and other PDEs; see, for instance, \cite{Komornik,Lasiecka,Lasiecka1,Lions,Lu,Mercado,Privat,Weiss,Yang5,Yao,Zuazua}. A fundamental framework for studying controllability is provided by the Hilbert Uniqueness Method (HUM) introduced by J.-L. Lions, which establishes a duality between controllability and observability. More precisely, under the appropriate functional-analytic framework, the controllability of a PDE can be characterized by an observability inequality for the corresponding adjoint system. Therefore, the study of controllability naturally leads to the study of observability inequalities.
	
	The study of controllability for degenerate partial differential equations (PDEs) has developed considerably since the 2000s. Over the past two decades, a substantial body of work has been devoted to the one-dimensional case; see, for example, \cite{Akil,Alabau,Alabau1,Bai,Buffe,Cannarsa,Cannarsa1,Fragnelli,Fragnelli1,Gueye,Guo}. Several results have also been established for two-dimensional degenerate PDEs; see, for instance, \cite{Beauchard,Cannarsa2}. More recently, controllability problems for certain higher-dimensional degenerate PDEs have been investigated in \cite{Guo,Yang8,Yang,Yang6,Yang4,Yang5,Yang1,Yang2,Yang3,Yang7}. 
	
	The methods employed in these works include Carleman estimates, multiplier methods, shape design and approximation techniques, as well as separation-of-variables arguments. These approaches provide different ways to overcome the analytical difficulties caused by the degeneracy of the differential operator.
	
	Shape design is a classical and well-established research area in mathematics, with a long history and broad connections to analysis, geometry, and applied sciences. It encompasses a variety of fundamental problems, including isoperimetric inequalities, the optimization of eigenvalues and spectral gaps, and the determination of optimal domains for variational and PDE-related quantities. Shape optimization also plays an important role in applications arising in science, engineering, and industry. A substantial body of literature has been devoted to these problems; see, for example, \cite{Buttazzo,Chenais,Greco,Guo2,Guo1,Guo,He,Henrot,Privat,Tiba,Wang,Wang1}. More recently, shape optimization and domain perturbation methods have been employed in the study of controllability and observability for PDEs; see \cite{Guo,Yang6,Yang2,Yang3,Yang7}. In the context of degenerate PDEs, shape design can also serve as an approximation tool. By modifying the domain so as to remove a thin neighborhood of the degenerate set, one obtains a family of uniformly elliptic problems on perturbed domains. This provides a natural framework for transferring analytical properties from the nondegenerate approximating problems to the original degenerate problem.

	 The Schrödinger equation is a fundamental equation in mathematical physics. The controllability and observability of both uniformly elliptic and degenerate Schrödinger equations have been extensively studied; see, for example, \cite{Fragnelli1,Privat,Weiss,Zuazua}.
	 
	 In this paper, we employ a shape design method to establish a boundary observability inequality for a two-dimensional degenerate Schrödinger equation. Our approach is different from those developed in \cite{Yang,Yang8}, where Carleman estimates are the main tools for deriving observability results. Instead, our method is closer to the strategy used in \cite{Yang7}, where shape design and approximation techniques are employed.
	 
	 However, there are essential differences between the present work and \cite{Yang7}. The Schrödinger equation possesses dispersive characteristics and exhibits analytical difficulties that are different from those of parabolic and hyperbolic equations. Although it shares some features with both types of equations, the derivation of observability inequalities for Schrödinger equations remains a challenging problem.

	In this paper, we study the following degenerate Schrödinger equation:
	\begin{equation}\label{09.15.1}
		\begin{cases}
			\mathrm{i}\partial_t\varphi-\operatorname{Div}(A\nabla\varphi)=0,
			&\text{in } Q,\\
			\varphi=0,
			&\text{on } \Sigma,\\
			\varphi(0)=\varphi^0,
			&\text{in } \Omega.
		\end{cases}
	\end{equation}
	Here, $
	\Omega=(-1,1)\times(0,1), Q=\Omega\times(0,T)$, 
	where \(T>0\), and $
	\Sigma=\partial\Omega\times(0,T)$. 
	The diffusion matrix is given by 
	$$
	A=\operatorname{diag}(1,w), \mbox{ and }w=x_2^\alpha \mbox{ and }  \alpha\in(0,1),
	$$ 
	where \(\alpha\) is a fixed constant. The degeneracy occurs on the boundary part $
	\Gamma_2^0=(-1,1)\times\{0\}$. 
	The initial datum satisfies $
	\varphi^0\in H_0^1(\Omega;w)$. The weighted Sobolev space $H_0^1(\Omega;w)$ will be defined in Section \ref{S2}. We introduce the associated elliptic operator
	\begin{equation}\label{09.15.4}
		\mathcal A\varphi=-\Div(A\nabla\varphi).
	\end{equation}
	
	For later use, we denote the boundary components by
	\begin{equation*}
		\Gamma_1^{-1}=\{-1\}\times(0,1),\quad
		\Gamma_1^{1}=\{1\}\times(0,1),
	\end{equation*}
	and
	\begin{equation*}
		\Gamma_2^i=(-1,1)\times\{i\},\quad i=0,1.
	\end{equation*}
	
	The two-dimensional setting considered in this work is chosen for the convenience of presentation. The same framework is expected to extend to higher-dimensional domains with diffusion matrices of the form 
	$$
	A=\operatorname{diag}(1,\ldots,1,x_N^\alpha),
	$$ 
	which correspond to uniformly elliptic behavior in the tangential directions and a one-dimensional degeneracy in the normal direction. The essential analytical difficulties, including the weighted energy estimates, spectral approximation, and boundary observability arguments, are already present in the two-dimensional case.

We organize the paper as follows. In Section \ref{S2}, we present some preliminary results, including the construction of the relevant energy spaces, the spectral properties of the operator \(\mathcal A\), the method of separation of variables, and the existence and uniqueness of weak solutions.

In Section \ref{S3}, we introduce a shape design method to construct a family of uniformly   Schrödinger equations that approximate the original degenerate equation. We also investigate the spectral properties of the operator \(\mathcal A^{(2)}\), which play an essential role in the approximation procedure and in the analysis of the corresponding spectral decomposition. The main approximation result is given in Theorem \ref{09.17.T2}, together with its consequence in Corollary \ref{09.20.C1}.

In Section \ref{S4}, we first establish an observability inequality for the uniformly elliptic Schrödinger equations obtained through the shape design procedure. We then pass to the limit in the approximation and derive the boundary observability inequality for the original degenerate Schrödinger equation. The main observability result is stated in Theorem \ref{09.21.T3}.

In the Appendix \ref{SA}, we provide the proofs of the existence and uniqueness of weak solutions for both the original degenerate Schrödinger equation and the uniformly   approximating equations.

The main results of this paper are Theorem \ref{09.17.T2}, Corollary \ref{09.20.C1}, and Theorem \ref{09.21.T3}.

\section{Preliminary}\label{S2}

In this paper, all the spaces are defined over the complex field $\mathbb{C}$; hence, the inner products are complex-valued. If $
u=u_1+\im u_2$, 
where $u_1$ and $u_2$ are real-valued functions, then $
\overline{u}=u_1-\im u_2$ 
denotes the complex conjugate of $u$, where $\im=\sqrt{-1}$ is the imaginary unit. Moreover, $\operatorname{Re}u$ and $\operatorname{Im}u$ denote the real and imaginary parts of the complex-valued function $u$, respectively.

\subsection{Energy spaces}\label{S2.1}

Let 
\begin{equation*}
	L^2(\Om)=L^2(\Om;\Cc)=\left\{u:\Om\ra\Cc\colon \int_\Om |u|^2\df x<+\iy\right\}. 
\end{equation*}
Define
\begin{equation*}
	H^1(\Om;w)=\left\{u\in L^2(\Om)\colon \nabla u\cdot A\nabla \ol u\in L^1(\Om)\right\}. 
\end{equation*}
Its inner product and norm are defined by 
\begin{equation*}
	(u,v)_{H^1(\Om;w)}=\int_\Om \left(u\ol v+\nabla u\cdot A\nabla \ol v\right)\df x,\quad \|u\|_{H^1(\Om;w)}=(u,u)_{H^1(\Om;w)}^\f{1}{2}. 
\end{equation*}
Set 
\begin{equation*}
	H_0^1(\Om;w)= \mbox{the closure of } C_0^\iy(\Om) \mbox{ in } H^1(\Om;w). 
\end{equation*}
Denote $H^{-1}(\Om;w)$ the dual space of $H_0^1(\Om;w)$ with pivot space $L^2(\Om)$. 

Now, we define
\begin{equation*}
	H^2(\Om;w)=\left\{u\in H^1(\Om;w)\colon \mcA u\in L^2(\Om)\right\}. 
\end{equation*}
Its inner product and norm are defined by 
\begin{equation*}
	(u,v)_{H^2(\Om;w)}=(u,v)_{H^1(\Om;w)}+(\mcA u, \mcA v)_{L^2(\Om)},\quad \|u\|_{H^2(\Om;w)}=(u,u)_{H^2(\Om;w)}^\f{1}{2}. 
\end{equation*}
We emphasize that the space $H^2(\Omega;w)$ introduced above is an operator-adapted weighted space and, in general, it is different from the classical Sobolev space $H^2(\Omega)$.

It is obvious that the spaces 
\begin{equation*}
	H_0^1(\Om;w),\quad H^1(\Om;w), \quad H^2(\Om;w)
\end{equation*}
are Hilbert spaces. 

We denote 
\begin{equation*}
	D(\mcA)=H_0^1(\Om;w)\cap H^2(\Om;w). 
\end{equation*}

\begin{lemma}\label{09.15.L1}
	For all $u\in H_0^1(\Omega;w)$, we have
	\begin{equation}\label{09.15.2}
		\int_\Omega x_2^{\alpha-2}|u|^2\,dx
		\leq
		\frac{4}{(1-\alpha)^2}
		\int_\Omega x_2^\alpha |\partial_{x_2}u|^2\,dx .
	\end{equation}
	Moreover,
	\begin{equation}\label{09.15.3}
		\int_\Omega |u|^2\,dx
		\leq
		\frac{4}{(1-\alpha)^2}
		\int_\Omega \nabla u\cdot A\nabla\overline u\,dx .
	\end{equation}
\end{lemma}

\begin{proof}
	We first prove the inequality for real-valued functions.  By the weighted Hardy inequality
	\cite[Lemma 2.2]{Yang7}, \cite[Lemma 2.2]{Yang6}, or
	\cite[Lemma 2.1]{Yang}, we have
	$$
	\int_\Omega x_2^{\alpha-2}u^2\,dx
	\leq
	\frac{4}{(1-\alpha)^2}
	\int_\Omega x_2^\alpha
	|\partial_{x_2}u|^2\,dx .
	$$
	
	Now, let $
	u=u_1+\im  u_2$, 
	where $u_1$ and $u_2$ are real-valued functions. Applying the above
	inequality to $u_1$ and $u_2$, respectively, we obtain
	$$
	\begin{aligned}
		\int_\Omega x_2^{\alpha-2}|u|^2\,dx
		&=
		\int_\Omega x_2^{\alpha-2}
		(u_1^2+u_2^2)\,dx\\
		&\leq
		\frac{4}{(1-\alpha)^2}
		\int_\Omega x_2^\alpha
		\left(
		|\partial_{x_2}u_1|^2
		+
		|\partial_{x_2}u_2|^2
		\right)\,dx =
		\frac{4}{(1-\alpha)^2}
		\int_\Omega x_2^\alpha
		|\partial_{x_2}u|^2\,dx .
	\end{aligned}
	$$
	Hence, \eqref{09.15.2} holds.
	
	Finally, since $0<x_2<1$ and $\alpha-2<0$, we have $
	1\leq x_2^{\alpha-2}$. 
	Therefore,
	$$
	\begin{aligned}
		\int_\Omega |u|^2\,dx
		&\leq
		\int_\Omega x_2^{\alpha-2}|u|^2\,dx \leq
		\frac{4}{(1-\alpha)^2}
		\int_\Omega x_2^\alpha
		|\partial_{x_2}u|^2\,dx \leq
		\frac{4}{(1-\alpha)^2}
		\int_\Omega \nabla u\cdot A\nabla\overline u\,dx .
	\end{aligned}
	$$
	This proves \eqref{09.15.3}.
\end{proof}

\begin{remark}\label{09.15.R1}
	Inequality \eqref{09.15.3} is a Poincar\'e-type inequality. Consequently, the norm
	\begin{equation*}
		\|u\|_{H_0^1(\Omega;w)}
		=
		\left(
		\int_\Omega
		\nabla u\cdot A\nabla\overline{u}\,dx
		\right)^{\frac12}
	\end{equation*}
	is an equivalent Hilbert norm induced by $H^1(\Omega;w)$ on $H_0^1(\Omega;w)$. In what follows, we use this norm on $H_0^1(\Omega;w)$.
\end{remark}

\begin{lemma}\label{09.15.L2}
	The embedding $
	H_0^1(\Omega;w)\hookrightarrow L^2(\Omega)$ 
	is compact.
\end{lemma}

\begin{proof}
	The compactness of the embedding follows from
	\cite[Lemma 2.4]{Yang7} or \cite[Lemma 2.3]{Yang}, which applies to the weighted Sobolev space
	$H_0^1(\Omega;w)$ with $w=x_2^\alpha$, $\alpha\in(0,1)$.
\end{proof}

\subsection{Spectrum}\label{S2.2}

Let
\begin{equation}\label{09.15.5}
	a(u,v)
	=
	\int_\Omega \nabla u\cdot A\nabla\overline{v}\,dx,
	\mbox{ for all } 
	u,v\in H_0^1(\Omega;w).
\end{equation}
Then the operator $\mathcal A$ is a positive self-adjoint operator in
$L^2(\Omega)$ with compact resolvent. Consequently, its spectrum consists
of a sequence of positive eigenvalues
\begin{equation}\label{09.16.14}
	0<\lambda_1\leq\lambda_2\leq\lambda_3\leq\cdots,
	\quad
	\lambda_k\ra +\infty \text{ as }k\to+\infty, \quad \la_1\geq \f{(1-\al)^2}{4}.
\end{equation}
We denote by $\Phi_n$, $n\in\mathbb N^*$, a normalized eigenfunction
associated with $\lambda_n$, that is,
\begin{equation*}
	\mathcal A\Phi_n=\lambda_n\Phi_n,
	\quad
	\|\Phi_n\|_{L^2(\Omega)}=1.
\end{equation*}
Since $\mathcal A$ has real coefficients, the eigenfunctions $\Phi_n$ can
be chosen to be real-valued. It is well known that $\{\Phi_n\}_{n\in\mathbb N^*}$ forms an orthonormal basis of $L^2(\Omega)$. Moreover, since
\begin{equation}\label{09.15.6}
a(\Phi_n,\Phi_m)
=
\lambda_n(\Phi_n,\Phi_m)_{L^2(\Omega)}=\la_n\de_{nm},
\end{equation}
the family $\{\Phi_n\}_{n\in\mathbb N^*}$ forms an orthogonal basis of $H_0^1(\Omega;w)$ with respect to the inner product $a(\cdot,\cdot)$.

\begin{definition}\label{09.15.D1}
	Let $\theta\in[0,+\infty)$. We define the domain of the fractional
	power $\mathcal A^\theta$ by
	\begin{equation}\label{09.15.7}
		D(\mathcal A^\theta)
		=
		\left\{
		u\in L^2(\Omega)\colon 
		\sum_{n=1}^{\infty}
		\lambda_n^{2\theta}
		\left|(u,\Phi_n)_{L^2(\Omega)}\right|^2
		<+\infty
		\right\}.
	\end{equation}
	The inner product and the associated norm on $D(\mathcal A^\theta)$
	are respectively defined by
	\begin{equation*}
		(u,v)_{D(\mathcal A^\theta)}
		=
		\sum_{n=1}^{\infty}
		\lambda_n^{2\theta}
		(u,\Phi_n)_{L^2(\Omega)}
		\overline{(v,\Phi_n)_{L^2(\Omega)}}, \mbox{ and } 
		\|u\|_{D(\mathcal A^\theta)}
		=
		\left(
		\sum_{n=1}^{\infty}
		\lambda_n^{2\theta}
		\left|(u,\Phi_n)_{L^2(\Omega)}\right|^2
		\right)^{\frac12}.
	\end{equation*}
\end{definition}

If $u\in D(\mcA^\theta)$ then $\mcA^\theta u=\sum_{n=1}^\iy \la_n^\theta (u, \Phi_n)_{L^2(\Om)}$ for all $\theta\geq 0$. And $D(\mcA^\theta)\hra D(\mcA^\eta)$ is compact for $\theta>\eta\geq 0$. Moreover, it is easily verified that $C_0^\iy(\Om)\s D(\mcA^\theta)$ for all $\theta \geq 0$. 

It follows from the spectral representation of $\mathcal A$ and the
associated closed quadratic form that
\begin{equation*}
	D(\mathcal A^0)=L^2(\Omega),\quad
	D(\mathcal A^{\frac12})=H_0^1(\Omega;w),\quad
	D(\mathcal A^1)=D(\mathcal A).
\end{equation*} 
More precisely, we have the following lemma. 

\begin{lemma}\label{09.15.L3}
	We have
	\begin{equation}\label{09.15.8}
		D(\mathcal A^{\frac12})
		=
		H_0^1(\Omega;w).
	\end{equation}
	Moreover, for every $u\in H_0^1(\Omega;w)$,
	\begin{equation}\label{09.15.9}
		\|\mathcal A^{\frac12}u\|_{L^2(\Omega)}^2
		=
		a(u,u)
		=
		\int_\Omega
		\nabla u\cdot A\nabla\overline{u}\,dx.
	\end{equation}
\end{lemma}

\begin{proof}
	For $u\in D(\mathcal A)$ and $v\in H_0^1(\Omega;w)$, integration
	by parts yields
	$$
	a(u,v)
	=
	\int_\Omega
	\nabla u\cdot A\nabla\overline v\,dx
	=
	(\mathcal Au,v)_{L^2(\Omega)}.
	$$
	Thus, the operator $\mathcal A$ defined above is the positive
	self-adjoint operator associated with the densely defined closed
	symmetric form $a$ on $L^2(\Omega)$.
	
	By the representation theorem for closed symmetric forms,
	$$
	D(\mathcal A^{\frac12})
	=
	D(a)
	=
	H_0^1(\Omega;w),
	$$
	and
	$$
	\|\mathcal A^{\frac12}u\|_{L^2(\Omega)}^2
	=
	a(u,u)
	=
	\int_\Omega
	\nabla u\cdot A\nabla\overline u\,dx.
	$$
	This proves the result.
\end{proof}

\begin{remark}\label{09.15.R2}
	Let $\Phi_n$, $n\in\mathbb N^*$, be the normalized eigenfunctions
	defined in \eqref{09.15.6}. Then
	$$
	\Phi_n\in D(\mathcal A^\theta), \text{ for all } \theta\geq0.
	$$
	In fact, $
	\|\Phi_n\|_{D(\mathcal A^\theta)}
	=
	\lambda_n^\theta$. 
\end{remark}

\begin{corollary}\label{09.15.C1}
	Let $\theta\in[1,+\infty)$. Then the operator
	\begin{equation*}
		\mathcal A:
		D(\mathcal A^\theta)
		\ra 
		D(\mathcal A^{\theta-1})
	\end{equation*}
	defines an isometry between these two spaces. More precisely,
	\begin{equation*}
		\|\mathcal Au\|_{D(\mathcal A^{\theta-1})}
		=
		\|u\|_{D(\mathcal A^\theta)}, 
		\text{ for all }u\in D(\mathcal A^\theta).
	\end{equation*}
\end{corollary}

\begin{proof}
	Let $u\in D(\mathcal A^\theta)$ and set $
	u_n=(u,\Phi_n)_{L^2(\Omega)}$. 
	By the spectral definition of the fractional powers of $\mathcal A$, $
	\mathcal Au
	=
	\sum_{n=1}^{\infty}
	\lambda_n u_n\Phi_n$. 
	Since $u\in D(\mathcal A^\theta)$, we have $
	\sum_{n=1}^{\infty}
	\lambda_n^{2\theta}|u_n|^2<+\infty$. 
	Therefore,
	$$
	\begin{aligned}
		\|\mathcal Au\|_{D(\mathcal A^{\theta-1})}^2
		&=
		\sum_{n=1}^{\infty}
		\lambda_n^{2(\theta-1)}
		|\lambda_n u_n|^2 =
		\sum_{n=1}^{\infty}
		\lambda_n^{2\theta}|u_n|^2 =
		\|u\|_{D(\mathcal A^\theta)}^2.
	\end{aligned}
	$$
	Hence
	$$
	\|\mathcal Au\|_{D(\mathcal A^{\theta-1})}
	=
	\|u\|_{D(\mathcal A^\theta)},
	$$
	which proves the result.
\end{proof}

\subsection{Method of separation of variables}\label{S2.3}

Denote
\begin{equation}\label{09.15.10}
	\Lambda_1=(-1,1),
	\quad
	\Lambda_2=(0,1).
\end{equation}
We introduce the one-dimensional operators
\begin{equation}\label{09.15.11}
	\mathcal A^{(1)}u=-\partial_{x_1x_1}u,
	\quad
	\mathcal A^{(2)}u
	=
	-\partial_{x_2}
	\left(
	x_2^\alpha\partial_{x_2}u
	\right), 
\end{equation}
where $\mathcal A^{(1)}$ acts only on the $x_1$-variable and $\mathcal A^{(2)}$ acts only on the $x_2$-variable. 
Then the operator $\mathcal A$ can be decomposed as
\begin{equation}\label{09.15.12}
	\mathcal A
	=
	\mathcal A^{(1)}+\mathcal A^{(2)}.
\end{equation}
This decomposition reflects the structure of the operator with uniformly elliptic behavior in the tangential direction and degeneracy only in the normal direction. 

Similar to the constructions in Sections~\ref{S2.1}-\ref{S2.2}, for the operator $\mcA^{(2)}$, we
define the weighted Sobolev spaces
$$
H^1(\Lambda_2;w),
\quad
H_0^1(\Lambda_2;w),
\quad
H^2(\Lambda_2;w),
$$
together with their corresponding inner products and norms. We emphasize that the space $H^2(\Lambda_2;w)$ introduced above is the operator-adapted weighted space associated with the degenerate operator $\mathcal A^{(2)}$. In general, it does not coincide with the classical Sobolev space $H^2(\Lambda_2)$.
 Here,
the weighted spaces on $\Lambda_2$ are defined with respect to the
weight $w(x_2)=x_2^\alpha$. 

On the other hand, since $\mathcal A^{(1)}$
is the classical one-dimensional Laplace operator, the spaces
$$
H^1(\Lambda_1),
\quad
H_0^1(\Lambda_1),
\quad
H^2(\Lambda_1)
$$
coincide with the corresponding classical Sobolev spaces.  For example, we define the weighted Sobolev space
\begin{equation*}
	H^1(\Lambda_2;w)
	=
	\left\{
	u\in L^2(\Lambda_2)\colon 
	w(x_2)|\pt_{x_2}u|^2\in L^1(\La_2)
	\right\}.
\end{equation*}
Its inner product and norm are respectively defined by
\begin{equation*}
	(u,v)_{H^1(\Lambda_2;w)}
	=
	\int_{\Lambda_2}
	\left(
	u\overline v
	+
	x_2^\alpha
	(\partial_{x_2}u)
	\overline{(\partial_{x_2}v)}
	\right)\,dx_2,\mbox{ and } 
	\|u\|_{H^1(\Lambda_2;w)}
	=(u,u)_{H^1(\La_2;w)}^\f{1}{2}.
\end{equation*}

Define 
\begin{equation*}
	D(\mcA^{(1)})=H_0^1(\La_1)\cap H^2(\La_1),\quad D(\mcA^{(2)})=H_0^1(\La_2;w)\cap H^2(\La_2;w). 
\end{equation*}
Note that all the spaces defined above are spaces over the complex field $\mathbb{C}$.

The one-dimensional weighted Hardy inequality below will be used repeatedly.

\begin{lemma}[{\cite[(2.1), p.~165]{Alabau}}]\label{09.15.L4}
	For all $u\in H_0^1(\La_2;w)$, we have 
	\begin{equation*}
		\int_{\La_2} x_2^{\al-2} |u|^2\df x_2\leq \f{4}{(1-\al)^2}\int_{\La_2} x_2^\al |\pt_{x_2}u|^2\df x_2. 
	\end{equation*}
\end{lemma}

\begin{remark}\label{09.15.R3}
	(1) By Lemma~\ref{09.15.L4}, for every
		$u\in H_0^1(\Lambda_2;w)$, we have
		\begin{equation*}
			\int_{\Lambda_2}|u|^2\,dx_2
			\leq
			\frac{4}{(1-\alpha)^2}
			\int_{\Lambda_2}
			x_2^\alpha|\partial_{x_2}u|^2\,dx_2.
		\end{equation*}
		Hence, the norm
		\begin{equation*}
			\|u\|_{H_0^1(\Lambda_2;w)}
			=
			\left(
			\int_{\Lambda_2}
			x_2^\alpha|\partial_{x_2}u|^2\,dx_2
			\right)^{\frac12}
		\end{equation*}
		is equivalent to the standard $H^1(\Lambda_2;w)$-norm on
		$H_0^1(\Lambda_2;w)$. In what follows, we use this norm on
		$H_0^1(\Lambda_2;w)$.
		
		(2)  The embedding
		\begin{equation*}
			H_0^1(\Lambda_2;w)\hookrightarrow L^2(\Lambda_2)
		\end{equation*}
		is compact. 
\end{remark}

Since the embeddings
\begin{equation*}
	H_0^1(\Lambda_1)\hookrightarrow L^2(\Lambda_1), \text{ and } 
	H_0^1(\Lambda_2;w)\hookrightarrow L^2(\Lambda_2)
\end{equation*}
are compact, the operators $\mathcal A^{(1)}$ and $\mathcal A^{(2)}$
defined in \eqref{09.15.11} are positive self-adjoint operators with
compact resolvents. They are associated with the sesquilinear forms
\begin{equation*}
	a^{(1)}(u,v)
	=
	\int_{\Lambda_1}
	(\partial_{x_1}u)
	\overline{(\partial_{x_1}v)}
	\,dx_1,
\end{equation*}
and
\begin{equation*}
	a^{(2)}(u,v)
	=
	\int_{\Lambda_2}
	x_2^\alpha
	(\partial_{x_2}u)
	\overline{(\partial_{x_2}v)}
	\,dx_2,
\end{equation*}
respectively. Consequently, for $i=1,2$, the spectrum of
$\mathcal A^{(i)}$ consists of a sequence of positive eigenvalues
\begin{equation}\label{09.18.8}
	0<\lambda_1^{(i)}
	\leq
	\lambda_2^{(i)}
	\leq
	\lambda_3^{(i)}
	\leq\cdots,
	\quad
	\lambda_k^{(i)}\ra+\infty
	\text{ as }k\to+\infty, \quad \la_1^{(2)}\geq \f{(1-\al)^2}{4}.
\end{equation}
We denote by $\Phi_n^{(i)}$, $n\in\mathbb N^*$, a normalized
eigenfunction of $\mathcal A^{(i)}$ associated with
$\lambda_n^{(i)}$, that is,
\begin{equation}\label{09.17.10}
	\mathcal A^{(i)}\Phi_n^{(i)}
	=
	\lambda_n^{(i)}\Phi_n^{(i)},
	\quad
	\|\Phi_n^{(i)}\|_{L^2(\Lambda_i)}=1.
\end{equation}
Since $\mathcal A^{(i)}$ has real coefficients, the eigenfunctions
$\Phi_n^{(i)}$ can be chosen to be real-valued. Moreover,
$\{\Phi_n^{(i)}\}_{n\in\mathbb N^*}$ forms an orthonormal basis of
$L^2(\Lambda_i)$.

In addition, for $u=\Phi_n^{(i)}$ and $v=\Phi_m^{(i)}$, we have
\begin{equation}\label{09.15.13}
	a^{(i)}(\Phi_n^{(i)},\Phi_m^{(i)})
	=
	\lambda_n^{(i)}
	(\Phi_n^{(i)},\Phi_m^{(i)})_{L^2(\Lambda_i)}=\la_n^{(i)}\de_{nm}.
\end{equation}
Therefore, $\{\Phi_n^{(1)}\}_{n\in\mathbb N^*}$ and
$\{\Phi_n^{(2)}\}_{n\in\mathbb N^*}$ form orthogonal bases of
$H_0^1(\Lambda_1)$ and $H_0^1(\Lambda_2;w)$, respectively, with
respect to the inner products $a^{(1)}(\cdot,\cdot)$ and
$a^{(2)}(\cdot,\cdot)$.

From \cite[Lemma 2, p.~2047]{Gueye}, \cite[(74), p.~198]{Cannarsa},
or \cite[(1.3), p.~2509]{Buffe}, there exists a constant $C>0$,
independent of $k$ and $\alpha$, such that
\begin{equation}\label{08.05.7}
	\sqrt{\lambda_{k+1}^{(2)}}
	-
	\sqrt{\lambda_k^{(2)}}
	\geq
	C(2-\alpha),
	\mbox{ for all }
	k\in\mathbb N^* \mbox{ and all }
	\alpha\in[0,1).
\end{equation}

The following lemma follows from the method of separation of variables.

\begin{lemma}[{\cite[Lemma 3.3]{Yang}}]\label{08.04.L6}
	We have
	\begin{equation*}
		H_0^1(\Omega;w)
		=
		H_0^1(\Lambda_1)\widehat\otimes H_0^1(\Lambda_2;w),
	\end{equation*}
	where $\widehat\otimes$ denotes the Hilbert tensor product in the sense of \cite{Yang}.
\end{lemma}

\begin{remark}\label{08.04.R3}
	By Lemma~\ref{08.04.L6}, the family $
	\{
	\Phi_n^{(1)}(x_1)\Phi_k^{(2)}(x_2)
	\}_{n,k\in\mathbb N^*}$ 
	forms an orthonormal basis of $L^2(\Omega)$. Moreover,
	\begin{equation*}
	\mathcal A
	\left(
	\Phi_n^{(1)}\Phi_k^{(2)}
	\right)
	=
	\left(
	\lambda_n^{(1)}+\lambda_k^{(2)}
	\right)
	\Phi_n^{(1)}\Phi_k^{(2)}.
	\end{equation*}
	Consequently, the spectrum of $\mathcal A$ consists of the
	eigenvalues
	\begin{equation*}
	\lambda_n^{(1)}+\lambda_k^{(2)},
	\mbox{ for all } 
	n,k\in\mathbb N^*,
	\end{equation*}
	counted with multiplicity. In what follows, we denote
	\begin{equation*}
		\lambda_{nk}
		=
		\lambda_n^{(1)}+\lambda_k^{(2)},
		\mbox{ for all } 
		n,k\in\mathbb N^*.
	\end{equation*}
	Thus, the eigenfunctions of $\mathcal A$ can be chosen in the form
	\begin{equation*}
	\Phi_{nk}(x_1,x_2)
	=
	\Phi_n^{(1)}(x_1)\Phi_k^{(2)}(x_2).
	\end{equation*}
\end{remark}

\begin{remark}\label{08.05.R1}
	The eigenpairs of $\mathcal A^{(1)}$ are explicitly given by
	\begin{equation*}
		\lambda_n^{(1)}
		=
		\frac{n^2\pi^2}{4},
		\quad
		\Phi_n^{(1)}(x_1)
		=
		\sin\left(
		\frac{n\pi(x_1+1)}{2}
		\right),
		\mbox{ for all } 
		n\in\mathbb N^*.
	\end{equation*}
	The eigenpairs of $\mathcal A^{(2)}$ are explicitly given by
	\begin{equation*}
		\lambda_n^{(2)}
		=
		(\kappa j_{\nu,n})^2,\quad 
		\Phi_n^{(2)}(x_2)
		=
		\frac{(2\kappa)^{\frac12}}
		{|J_\nu'(j_{\nu,n})|}
		x_2^{\frac12(1-\alpha)}
		J_\nu(j_{\nu,n}x_2^\kappa), \mbox{ for all } n\in\N^*, 
	\end{equation*}
	where
	\begin{equation*}
		\nu=\frac{1-\alpha}{2-\alpha},
		\quad
		\kappa=\frac{2-\alpha}{2}.
	\end{equation*}
	See \cite[(4.35) and (4.38), p.~2046-2047]{Gueye}.
	
	In particular,
	$$
	\lambda_n^{(1)}\sim n^2,  \text{ and } 
	\lambda_n^{(2)}\sim n^2 \text{ as }n\to+\infty,
	$$
	where the second asymptotic relation follows from the asymptotic
	behavior of the positive zeros $j_{\nu,n}$ of the Bessel function
	$J_\nu$.
\end{remark}

\subsection{Weak solution}\label{S2.4}

\begin{definition}\label{09.16.D1}
	Let $\vp^0\in H_0^1(\Om;w)$ and $f\in L^2(0,T; H_0^1(\Om;w))$. We call $\vp$ is a weak solution of the following equation 
	\begin{equation} \label{09.16.1}
		\begin{cases}
			\im\partial_{t}\vp+\mcA \vp=f, & \text{in } Q,\\
			\vp=0, & \text{on } \Sigma,\\
			\vp(0)=\vp^0,  & \text{in } \Omega
		\end{cases}
	\end{equation}
	with respect to $(\vp^0,f)$, 
	if 
	\begin{equation}\label{09.16.3}
		\vp\in  L^2(0,T; H_0^1(\Om;w)), \mbox{ with } \pt_t\vp\in L^2(0,T; H^{-1}(\Om;w))
	\end{equation}
	satisfies
	
	(i) for all $v\in H_0^1(\Om;w)$ and a.e.~$t\in [0,T]$, 
	\begin{equation}\label{09.16.2}
		\im \lg \pt_t\vp, v\rg_{H^{-1}(\Om;w), H_0^1(\Om;w)}+a(\vp, v)=(f,v)_{L^2(\Om)}, \mbox{ for all } v\in H_0^1(\Om;w), 
	\end{equation}
	
	(ii) $
		\vp(0)=\vp^0$. 
\end{definition}

\begin{theorem}\label{09.16.T1}
	Let $\vp^0\in H_0^1(\Om;w)$ and $f\in L^2(0,T; H_0^1(\Om;w))$. Then there exists a unique weak solution
	\begin{equation*}
		\vp\in C([0,T]; L^2(\Om))\cap L^\iy(0,T; H_0^1(\Om;w)), \mbox{ with } \pt_t\vp\in L^2(0,T; H^{-1}(\Om;w))
	\end{equation*} 
	to the equation \eqref{09.16.1} with respect to $(\vp^0,f)$ satisfying 
	\begin{equation}\label{09.16.4}
		\begin{split}
			&\esssup_{t\in [0,T]}\|\vp(t)\|_{H_0^1(\Om;w)}+\|\pt_t\vp\|_{L^2(0,T; H^{-1}(\Om;w))} \leq C\left(\|\vp^0\|_{H_0^1(\Om;w)}+\|f\|_{L^2(0,T; H_0^1(\Om;w))}\right), 
		\end{split}
	\end{equation}
	where the positive constant $C$ depends only on $\al$ and $T$. 
\end{theorem}

The proof of Theorem \ref{09.16.T1} is postponed to Appendix \ref{SA}.

\begin{remark}\label{09.16.R1}
	Assume that $f=0$ and let
	\begin{equation*}
		\vp^0
		=
		\sum_{n,k=1}^{\infty}
		\vp_{nk}^0
		\Phi_n^{(1)}
		\Phi_k^{(2)}
		\in H_0^1(\Om;w)\ \Lra \ \sum_{n,k=1}^\iy \la_{nk}|\vp_{nk}^0|^2<+\iy, 
	\end{equation*}
	where $
	\vp_{nk}^0
	=
	(\vp^0,
	\Phi_n^{(1)}
	\Phi_k^{(2)})_{L^2(\Om)}$ 
	for all $n,k\in\mathbb N^*$.
	Then the solution of \eqref{09.16.1} admits the spectral
	expansion
	\begin{equation}\label{09.16.11}
		\vp(t)
		=
		\sum_{n,k=1}^{\infty}
		e^{\im\lambda_{nk}t}
		\vp_{nk}^0
		\Phi_n^{(1)}
		\Phi_k^{(2)} .
	\end{equation}
\end{remark}

\section{Shape design}\label{S3} 

In this section, we introduce a shape design approximation by removing a thin neighborhood of the degenerate boundary. We show that the solution of the original degenerate Schrödinger equation can be approximated by solutions of Schrödinger equations associated with uniformly elliptic operators on the designed domains $\Omega_\delta$. And we assume $\de\in (0,\f{1}{8})$.

Denote 
\begin{equation*}
	\Om_\de=(-1,1)\ts (\de,1), \quad w_\de=w|_{\Om_\de}, \quad A_\de=A|_{\Om_\de}, 
\end{equation*} 
Denote by $\mathcal A_\delta$ the Dirichlet realization of $
-\Div(A_\delta\nabla \bullet)$ 
in $\Omega_\delta$, 
and $Q_\de=\Om_\de\ts (0,T)$, $\Sigma_\de=\pt\Om_\de\ts (0,T)$, and 
\begin{equation*}
	\Ga_{\de,1}^{i}=\{i\}\ts (\de,1),\ i=-1,1, \quad \Ga_{\de,2}^j=(-1,1)\ts \{j\},\ j=\de, 1. 
\end{equation*}

Similar to Section~\ref{S2.1}, we can define the spaces
$$
H^1(\Omega_\delta;w_\delta),\quad
H_0^1(\Omega_\delta;w_\delta),\quad
H^2(\Omega_\delta;w_\delta).
$$
For example, we define
$$
H^1(\Omega_\delta;w_\delta)
=
\left\{
u\in L^2(\Omega_\delta):
\int_{\Omega_\delta}
\nabla u\cdot A_\delta\nabla\overline u\,dx<+\infty
\right\}.
$$
Since $\delta>0$, the matrix $A_\delta$ is uniformly elliptic on
$\Omega_\delta$. Moreover, $
0<\delta^\alpha\leq x_2^\alpha\leq 1$ in $\Omega_\delta$. 
Therefore, the weighted Sobolev spaces are equivalent to the standard
Sobolev spaces, that is,
$$
H^1(\Omega_\delta;w_\delta)=H^1(\Omega_\delta),\quad 
H_0^1(\Omega_\delta;w_\delta)
=
H_0^1(\Omega_\delta), \quad 
H^2(\Omega_\delta;w_\delta)
=
H^2(\Omega_\delta).
$$
The last equivalence also follows from the standard elliptic regularity
for uniformly elliptic operators.

We  also introduce 
\begin{equation*}
	\La_2^\de=(\de,1). 
\end{equation*}
And define $H_0^1(\La_2^\de;w_\de), H^1(\La_2^\de;w_\de)$ and $H^2(\La_2^\de;w_\de)$ as above. Define 
\begin{equation}\label{09.18.1}
	\mcA_\de u =\mcA_\de ^{(1)} u+\mcA_\de^{(2)} u, \mbox{ with } \mcA_\de^{(1)}=-\pt_{x_1x_1}u \mbox{ and } \mcA_\de^{(2)}=-\pt_{x_2}(x_2^\al \pt_{x_2}u). 
\end{equation}

Since the embedding  $H_0^1(\La_2^\de;w_\de)=H_0^1(\La_2^\de)\hra L^2(\La_2^\de)$ is compact, then $\mcA_\de^{(2)}$ has a purely discrete simple spectrum
\begin{equation}\label{09.18.2}
	0<\la_{\de,1}^{(2)}\leq \la_{\de,2}^{(2)}\leq \la_{\de,3}^{(2)}\leq \cdots \ra+\iy. 
\end{equation}
And we denote $\Phi_{\de,n}^{(2)}$ the normalized $n$th eigenfunction of $\mcA_\de^{(2)}$ with respect to the $n$th eigenvalue  $\la_{\de,n}^{(2)}$. i.e., 
\begin{equation}\label{09.18.3}
	a_\de^{(2)}(\Phi_{\de, n}^{(2)}, \Phi_{\de, m}^{(2)})=\la_{\de,n}^{(2)}(\Phi_{\de,n}^{(2)}, \Phi_{\de, m}^{(2)})_{L^2(\La_2^\de)}=\la_{\de,n}^{(2)}\de_{nm}, \mbox{ for all } n,m\in\N^*. 
\end{equation}

From above, we get the following lemma by Lemma \ref{08.04.L6}.  

\begin{lemma}\label{08.04.L9}
	We have 
	\begin{equation*}
		H_0^1(\Om_\de;w_\de)=H_0^1({\La_1})\wh\ots H_0^1({\La_2^\de};w_\de). 
	\end{equation*} 
\end{lemma}

\begin{remark}\label{08.04.R4}
	From Lemma \ref{08.04.L9}, we obtain  
	$\{\Phi_n^{(1)}\Phi_{\de,k}^{(2)}\}_{n,k\in\N^*}$ is an  orthonormal basis of $L^2(\Om_\de)$, and $\la_n^{(1)}+\la_{\de,k}^{(2)}\ (n,k\in\N^*)$ form the spectrum of $\mcA_\de$. Hence, for all $m\in\N^*$, there exist $n,k\in\mathbb N^*$ such that  $\la_m^\de=\la_n^{(1)}+\la_{\de,k}^{(2)}$, and $\Phi_{\de, m}=\Phi_n^{(1)}\Phi_{\de,k}^{(2)}$ for some $n,k\in\N^*$. Throughout in this paper, we denote
	\begin{equation*}
		\la_{nk}^\de
		=
		\lambda_n^{(1)}+\lambda_{\de,k}^{(2)}.
	\end{equation*} 
\end{remark}

\begin{definition}\label{09.17.D2}
	Let  $u_\de\in H_0^1(\Om_\de;w_\de)$. For each $\ga\in [0,\de)$, define
	\begin{equation*}
		E_\de^\ga u_\de(x)=
		\begin{cases} 
			u_\de(x), &x\in \Om_\de,\\
			0, &x\in \Om_\ga-\Om_\de, 
		\end{cases}
	\end{equation*}
	then the zero extension
	$E_\de^\ga u_\de$ belongs to
	$H_0^1(\Om_\ga;w_\ga)$ and 
	$$
	\|E_\de^\ga u_\de\|_{H_0^1(\Om_\ga;w_\ga)}
	=
	\|u_\de\|_{H_0^1(\Om_\de;w_\de)} .
	$$
	
	Similarly, for $u_\de\in H_0^1(\La_2^\de;w_\de)$, we define
	$E_\de^\ga u_\de$ on $\Lambda_2^\ga$ by
	$$
	E_\de^\ga u_\de(x)=
	\begin{cases}
		u_\de(x),&x\in\Lambda_2^\de,\\
		0,&x\in\Lambda_2^\ga-\Lambda_2^\de .
	\end{cases}
	$$
	Then
	$E_\de^\ga u_\de\in H_0^1(\La_2^\ga;w_\ga)$ and $
	\|E_\de^\ga u_\de\|_{H_0^1(\La_2^\ga;w_\ga)}
	=
	\|u_\de\|_{H_0^1(\La_2^\de;w_\de)}$. 
\end{definition}

\begin{remark}\label{09.17.R1}
	Let $\ga\in[0,\de)$. From Definition \ref{09.17.D2}, by identifying
	each function with its zero extension through \begin{equation*}
		E_\de^\ga: H_0^1(\Om_\de;w_\de)\ra H_0^1(\Om_\ga;w_\ga), 
	\end{equation*}
	we can regard
	$H_0^1(\Om_\de;w_\de)$ as a closed subspace of
	$H_0^1(\Om_\ga;w_\ga)$. Similarly, we can regard
	$H_0^1(\La_2^\de;w_\de)$ as a closed subspace of
	$H_0^1(\La_2^\ga;w_\ga)$.
	
	Hence, all the weighted Sobolev spaces and related results established
	in Sections \ref{S2.1}-\ref{S2.3} remain valid after replacing
	$\Om$ and $\La_2$ by $\Om_\de$ and $\La_2^\de$, respectively. 
\end{remark}

\subsection{Weak solution}

In this subsection, we study the weak well-posedness of the uniformly   Schr\"odinger equation obtained by the shape design approximation, and we assume  $\de\in (0,\f{1}{8})$. 

Now, we consider the following system 
\begin{equation}\label{09.16.27}
	\begin{cases}
		\im \pt_t\vp_\de+\mcA_\de \vp_\de=f_\de, &\mbox{in }Q_\de, \\
		\vp_\de=0, &\mbox{on }\Sigma_\de, \\
		\vp_\de(0)=\vp_\de^0, &\mbox{in }\Om_\de, 
	\end{cases}
\end{equation}
where $
\vp_\de^0\in H_0^1(\Om_\de;w_\de)$, and $
f_\de\in L^2(0,T;H_0^1(\Om_\de;w_\de))$. 

\begin{definition}\label{09.17.D1}
	Let $\vp_\de^0\in H_0^1(\Om_\de;w_\de)$ and $f_\de\in L^2(0,T; H_0^1(\Om_\de;w_\de))$. We call $\vp_\de$ is a weak solution of the   equation \eqref{09.16.27} 
	with respect to $(\vp_\de^0,f_\de)$, 
	if 
	\begin{equation}\label{09.17.2}
		\vp_\de\in  L^2(0,T; H_0^1(\Om_\de;w_\de)), \mbox{ with } \pt_t\vp_\de\in L^2(0,T; H^{-1}(\Om_\de;w_\de))
	\end{equation}
	satisfies
	
	(i) for all $v\in H_0^1(\Om_\de;w_\de)$ and a.e.~$t\in [0,T]$, 
	\begin{equation}\label{09.17.3}
		\im \lg \pt_t\vp_\de, v\rg_{H^{-1}(\Om_\de;w_\de), H_0^1(\Om_\de;w_\de)}+a(\vp_\de, v)=(f_\de,v)_{L^2(\Om_\de)}, \mbox{ for all } v\in H_0^1(\Om_\de;w_\de), 
	\end{equation}
	
	(ii) $
	\vp_\de(0)=\vp_\de^0$. 
\end{definition}

\begin{theorem}\label{09.17.T1}
	Let $\vp_\de^0\in H_0^1(\Om_\de;w_\de)$ and $f_\de\in L^2(0,T; H_0^1(\Om_\de;w_\de))$. Then there exists a unique weak solution
	\begin{equation*}
		\vp_\de\in C([0,T]; L^2(\Om_\de))\cap L^\iy(0,T; H_0^1(\Om_\de;w_\de)), \mbox{ with } \pt_t\vp_\de\in L^2(0,T; H^{-1}(\Om_\de;w_\de))
	\end{equation*} 
	to the equation \eqref{09.16.27} with respect to $(\vp_\de^0,f_\de)$ satisfying 
	\begin{equation}\label{09.17.4}
		\begin{split}
			&\esssup_{t\in [0,T]}\|\vp_\de(t)\|_{H_0^1(\Om_\de;w_\de)}+\|\pt_t\vp_\de\|_{L^2(0,T; H^{-1}(\Om_\de;w_\de))}\\
			&\leq C\left(\|\vp_\de^0\|_{H_0^1(\Om_\de;w_\de)}+\|f_\de\|_{L^2(0,T; H_0^1(\Om_\de;w_\de))}\right), 
		\end{split}
	\end{equation}
	where the positive constant $C$ depends only on $\al$ and $T$. 
\end{theorem}

The proof of Theorem \ref{09.17.T1} is postponed to Appendix \ref{SA}.

 Similar to Remark \ref{09.16.R1}, we obtain the following remark. 

\begin{remark}\label{09.16.R2}
	Let
	\begin{equation*}
		\vp_\de^0
		=
		\sum_{n,k=1}^{\infty}
		\vp_{\de, nk}^0
		\Phi_n^{(1)}
		\Phi_{\de,k}^{(2)}
		\in H_0^1(\Om_\de;w_\de),
	\end{equation*}
	where $
	\vp_{\de, nk}^0
	=
	(\vp_\de^0,
	\Phi_n^{(1)}
	\Phi_{\de,k}^{(2)})_{L^2(\Om_\de)}$ 
	for all $n,k\in\mathbb N^*$.
	Then the solution of \eqref{09.16.27} admits the spectral
	expansion
	\begin{equation}\label{09.16.28}
		\vp_\de(t)
		=
		\sum_{n,k=1}^{\infty}
		e^{\im\lambda_{nk}^\de t}
		\vp_{\de,nk}^0
		\Phi_n^{(1)}
		\Phi_{\de,k}^{(2)} .
	\end{equation}
\end{remark}

\begin{lemma}\label{09.17.L1}
	Let $\vp_\de^0\in C_0^\infty(\Om_\de)\s D(\mcA_\de)$ and let $\vp_\de$
	be the solution defined in \eqref{09.16.28}. Then
	$$
	\vp_\de\in
	L^2(0,T;H_0^1(\Om_\de;w_\de)\cap H^2(\Om_\de)), \mbox{ and } \pt_t\vp_\de \in L^2(Q_\de). 
	$$
\end{lemma}

\begin{proof}
	We only need to show $\vp_\de\in L^2(0,T; H^2(\Om_\de))$, since $\vp_\de\in L^2(0,T; H_0^1(\Om_\de;w_\de))$. 
	
	Indeed, from \eqref{09.17.10} and \eqref{09.15.13}, we obtain 
	\begin{equation*}
		\begin{split}
			\iint_{Q_\de} |\pt_{x_1x_1}\vp_\de|^2\df x\df t
			&=\iint_{Q_\de} \left|\sum_{n,k=1}^\iy e^{\im \la_{nk}^\de t}\vp_{\de,nk}^0 \left(\pt_{x_1x_1}\Phi_n^{(1)}\right)\Phi_{\de,k}^{(2)} \right|^2\df x\df t\\
			&= \iint_{Q_\de} \left|\sum_{n,k=1}^\iy e^{\im \la_{nk}^\de t} \vp_{\de, nk}^0\la_n^{(1)}\Phi_n^{(1)}\Phi_{\de,k}^{(2)}\right|^2\df x\df t\\
			&\leq T\sum_{n,k=1}^\iy |\vp_{\de, nk}^0|^2(\la_n^{(1)})^2\leq T \sum_{n,k=1}^\iy |\vp_{\de,nk}^0|^2(\la_{nk}^{\de})^2=T\|\mcA_\de\vp_\de^0\|_{L^2(\Om_\de)}^2, 
		\end{split}
	\end{equation*}
	hence, $\pt_{x_1x_1}\vp_\de\in L^2(Q_\de)$. 
	
	From \eqref{09.17.10} and \eqref{09.15.13}, we obtain 
	\begin{equation*}
		\begin{split}
			&\iint_{Q_\de} x_2^\al|\pt_{x_1x_2}\vp_\de|^2\df x\df t\\
			&=\iint_{Q_\de}x_2^\al\left|\sum_{n,k=1}^\iy e^{\im \la_{nk}^\de t} \vp_{\de, nk}^0 (\pt_{x_1}\Phi_n^{(1)})(\pt_{x_2}\Phi_{\de,k}^{(2)})\right|^2\df x\df t\\
			&=\left(\sum_{n,k=1}^\iy e^{\im \la_{nk}^\de t}\vp_{\de,nk}^0 (\pt_{x_1}\Phi_n^{(1)})(x_2^\f{\al}{2}\pt_{x_2}\Phi_{\de,k}^{(2)}), \sum_{j,l=1}^\iy e^{\im \la_{jl}^\de t}\vp_{\de,jl}^0 (\pt_{x_1}\Phi_j^{(1)})(x_2^\f{\al}{2}\pt_{x_2}\Phi_{\de,l}^{(2)}) \right)_{L^2(Q_\de)}\\
			&\leq T \sum_{n,k=1}^\iy |\vp_{\de,nk}^0|^2 \la_n^{(1)}\la_{\de,k}^{(2)}\leq \f{1}{4}T\sum_{n,k=1}^\iy |\vp_{\de,nk}^0|^2(\la_{nk}^\de)^2=\f{1}{4}T \|\mcA_\de\vp_\de^0\|_{L^2(Q_\de)}^2, 
		\end{split}
	\end{equation*}
	then $\|\pt_{x_1x_2}\vp_\de\|_{L^2(Q_\de)}\leq \f{\sqrt{T}}{2}\de^{-\f{\al}{2}}\|x_2^\f{\al}{2}\pt_{x_1x_2}\vp_\de\|_{L^2(Q_\de)}<+\iy$.
	
	From \eqref{09.17.10} and \eqref{09.15.13}, we get 
	\begin{equation*}
		\begin{split}
			&\iint_{Q_\de} x_2^{2\al} |\pt_{x_2x_2}\vp_\de|^2\df x\df t\\
			&=\iint_{Q_\de} \left|\sum_{n,k=1}e^{\im \la_{nk}^\de t}\vp_{\de, nk}^0 \Phi_n^{(1)}(x_2^\al \pt_{x_2x_2}\Phi_{\de,k}^{(2)})\right|^2\df x\df t\\
			&=\iint_{Q_\de}\left|\sum_{n,k=1}e^{\im \la_{nk}^\de t}\vp_{\de, nk}^0 \Phi_n^{(1)}\left(-\mcA_\de \Phi_{\de,k}^{(2)}-\al x_2^{\al-1} \pt_{x_2}\Phi_{\de,k}^{(2)}\right)\right|^2\df x\df t\\
			&\leq 2T \sum_{n,k=1}^\iy |\vp_{\de, nk}^0|^2 (\la_{\de,k}^{(2)})^2+2\al^2 \iint_{Q_\de} x_2^{\al-2}\left|\sum_{n,k=1}e^{\im \la_{nk}^\de t}\vp_{\de, nk}^0 \Phi_n^{(1)} (x_2^\f{\al}{2} \pt_{x_2}\Phi_{\de,k}^{(2)})\right|^2\df x\df t\\
			&\leq 2T\|\mcA_\de\vp_\de^0\|_{L^2(\Om_\de)}^2+2\al^2\de^{\al-2}T \sum_{n,k=1}^\iy |\vp_{\de,nk}^0|^2 \la_{\de,k}^{(2)}, 
		\end{split}
	\end{equation*} 
	hence, $\|\pt_{x_2x_2}\vp_\de\|_{L^2(Q_\de)}\leq C \de^{-\f{\al}{2}-1}\sqrt{T} \|x_2^\al \pt_{x_2x_2}\vp_\de\|_{L^2(Q_\de)}<+\iy$  since $\la_{\de,k}^{(2)}\leq \la_{nk}^\de$, where the positive constant $C$ depends only on $\al$. 
	
	Finally, we have 
	\begin{equation*}
		\begin{split}
			\iint_{Q_\de} |\pt_t\vp_\de|^2\df x\df t
			&=\iint_{Q_\de}\left|\sum_{n,k=1}^\iy (\pt_te^{\im \la_{nk}^\de t})\vp_{\de,nk}^0 \Phi_n^{(1)}\Phi_{\de,k}^{(2)}\right|^2\df x\df t\\
			&=\iint_{Q_\de} \left|\sum_{n,k=1}^\iy (\im \la_{nk}^\de )e^{\im \la_{nk}^\de t}\vp_{\de,nk}^0 \Phi_n^{(1)}\Phi_{\de,k}^{(2)}\right|^2\df x\df t\\
			&\leq T \sum_{n,k=1}^\iy |\vp_{\de,nk}^0|^2 (\la_{nk}^\de)^2=T\|\mcA_\de\vp_\de^0\|_{L^2(\Om_\de)}^2, 
		\end{split}
	\end{equation*}
	hence $\pt_t\vp_\de\in L^2(Q_\de)$.  
	We prove the lemma. 
\end{proof}

\begin{remark}\label{09.17.R3}
	From the proof of Lemma \ref{09.17.L1}, we observe that the
	$H^2$-regularity of $\vp_\delta$ relies on the fact that the degeneracy is removed by the shape design (i.e.,  $\mathcal A_\delta$ in $\Omega_\delta$ for $\delta>0$).
	However, for the original degenerate problem \eqref{09.15.1},
	a weak solution $\vp$ does not necessarily belong to
	$L^2(0,T;H^2(\Omega))$ in general. 
	Indeed, the degeneracy of the coefficient $x_2^\alpha$ prevents
	the standard elliptic $H^2$-regularity estimate from being applied
	up to the degenerate boundary $\Gamma_2^0$.
\end{remark}

\subsection{Spectral approximation of $\mcA^{(2)}$}

In this subsection, we regard the case $\delta=0$ as the original
degenerate problem \eqref{09.15.1} in Sections \ref{S2.1}-\ref{S2.3}, while the cases
$\delta\in(0,\frac18)$ correspond to the uniformly elliptic shape
design problems \eqref{09.16.27}.

Let $\de\in [0,\f{1}{8})$.  Define
\begin{equation*}
	r_\de: L^2(\La_2)\ra L^2(\La_2), \quad f\mapsto E_\de^0 u_\de^f, 
\end{equation*}
where $u_\de^f$ is the weak solution of the following degenerate elliptic equation 
\begin{equation}\label{09.17.9}
	\begin{cases}
		\mcA_\de^{(2)}u_\de^f=f|_{\La_2^\de}, &\mbox{in }\La_2^\de, \\
		u_\de^f(x)=0, &\mbox{on } \pt\La_2^\de. 
	\end{cases}
\end{equation}
Then, for each $\de\in [0,\f{1}{8})$, from the Lax-Milagram, there exists a unique weak solution $u_\de^f\in H_0^1(\La_2^\de;w_\de)$ to the equation \eqref{09.17.9}, such that 
\begin{equation*}
	\int_{\La_2^\de} w_\de (\pt_{x_2}u_\de ^f)\ol{(\pt_{x_2} v)}\df x_2=\int_{\La_2^\de} f\ol v\df x_2, \mbox{ for all } v\in H_0^1(\La_2^\de;w_\de). 
\end{equation*}
Moreover, from Lemma \ref{09.15.L4} and Remark \ref{09.17.R1}, we get 
\begin{equation}\label{09.18.5}
	\begin{split} 
	\|r_\de f\|_{L^2(\La_2)}
	&=\|E_\de^0 u_\de^f\|_{L^2(\La_2)}=\|u_\de^f\|_{L^2(\La_2^\de)}\\
	&\leq C\|u_\de^f\|_{H_0^1(\La_2^\de;w_\de)}\leq C\|f|_{\Om_\de}\|_{L^2(\La_2^\de)}\leq C\|f\|_{L^2(\La_2)}, 
	\end{split} 
\end{equation}
where the positive constant $C$ depends only on $\al$. Hence $r_\de: L^2(\La_2)\ra L^2(\La_2)$ is a linear bounded operator.  Moreover, since   the embedding
$
H_0^1(\Lambda_2;w)\hookrightarrow L^2(\Lambda_2)
$
is compact, $r_\delta$ is compact. Hence, $r_\de\ (0\leq \de< \f{1}{8})$ is a compact, self-adjoint and nonnegative operator on $L^2(\La_2)$. 

\begin{lemma}\label{09.17.L2}
	Let $f\in L^2(\La_2)$. Then
	\begin{equation*}
		r_\de(f)\ra r_0(f) \mbox{ strongly in } H_0^1(\La_2;w) \mbox{ as } \de \ra 0^+. 
	\end{equation*}
\end{lemma}

\begin{proof}
	For each $\de\in [0,\f{1}{8})$, from Definition \ref{09.17.D2},  we have 
	\begin{equation}\label{09.18.9}
		\|r_\de(f)\|_{H_0^1(\La_2;w)}=\|E_\de^0 u_\de^f\|_{H_0^1(\La_2;w)}=\|u_\de^f\|_{H_0^1(\La_2^\de;w_\de)}\leq C\|f|_{\La_2^\de}\|_{L^2(\La_2^\de)}\leq C\|f\|_{L^2(\La_2)}, 
	\end{equation}
	where the positive constants $C$ depends only on $\al$.  Hence, up to a subsequence, there exists $u_0\in H_0^1(\La_2;w)$ such that 
	\begin{equation}\label{09.17.11}
		r_\de(f)\ra u_0 \mbox{ weakly in }H_0^1(\La_2;w), \mbox{ and } r_\de(f)\ra u_0 \mbox{ strongly in } L^2(\La_2). 
	\end{equation}
	
	For each $v\in C_0^\iy(\La_2)$, for all $\de\in (0, \dist(\supp v, \{0\}))$,  we have 
	\begin{equation*}
		\int_{\La_2} w(\pt_{x_2}u_\de^f)\ol{(\pt_{x_2}v)}\df x_2=\int_{\La_2^\de} w_\de (\pt_{x_2} u_\de^f)\ol{(\pt_{x_2} v)}\df x_2=\int_{\La_2^\de} f|_{\La_2^\de} \ol v\df x_2=\int_{\La_2}f\ol v\df x_2, 
	\end{equation*}
	and then, from \eqref{09.17.11},  
	\begin{equation*}
		\int_{\La_2} w(\pt_{x_2}u_0)\ol{(\pt_{x_2}v)}\df x_2=\int_{\La_2} f\ol v\df x_2. 
	\end{equation*}
	Hence, $u_0$ is a weak solution of \eqref{09.17.9} with $\de=0$. Which implies 
	\begin{equation}\label{09.17.12}
		u_0=r_0(f). 
	\end{equation}
	
	Finally, from \eqref{09.17.11}, we obtain 
	\begin{equation*}
		\int_{\La_2} w|\pt_{x_2}r_\de(f)|^2\df x_2=\int_{\La_2} f\ol{r_\de(f)}\df x_2\ra \int_{\La_2} f\ol{r_0(f)}\df x_2=\int_{\La_2} w|\pt_{x_2}r_0(f)|^2\df x_2. 
	\end{equation*}
	Combining this and \eqref{09.17.11}, and the classical result: if $u_n\ra u$ weakly in a Hilbert space $H$, and $\|u_n\|_H\ra \|u\|_H$, then $u_n\ra u$ strongly in $H$, we obtain 
	\begin{equation*}
		r_\de(f)\ra r_0(f) \mbox{ strongly in } H_0^1(\La_2;w). 
	\end{equation*}
	This proves the lemma. 
\end{proof}

\begin{remark}
	Lemma \ref{09.17.L2} may be viewed as a resolvent convergence
	result for the family of elliptic operators ${\mathcal A_\de^{(2)}}$.
	The convergence is obtained directly from the uniform energy
	estimates, the compact embedding, and the variational formulation.
	This type of resolvent convergence is closely related to the
	convergence of the associated quadratic forms.
\end{remark}

\begin{lemma}[{\cite[Corollaries XI.9.3 and XI.9.4]{Dunford}}]\label{08.04.L11}
	Let $r_1,r_2$ be compact, self-adjoint and nonnegative operators on a Hilbert space $H$. For every $m,n\geq 1$ we have 
	\begin{equation*}
		|\xi_n(r_1)-\xi_n(r_2)|\leq \|r_1-r_2\|_{\mcL(H)},  
	\end{equation*}
	where $\xi_n(r_i)$ is the $n$th eigenvalue of $r_i$ for $i=1,2$. 
\end{lemma}

\begin{lemma}\label{08.04.L10}
	For each $k\in\N^*$, we have 
	\begin{equation*}
		\la_{\de,k}^{(2)}\dra  \la_k^{(2)} \mbox{ as } \de\ra 0^+. 
	\end{equation*}
\end{lemma}

\begin{proof}
	We prove this lemma by the following steps. 
	
	{\it Step 1}. We have known that 
	\begin{equation*}
		\begin{split}
			\la_{\de,k}^{(2)}
			&=\min_{E_k\s H_0^1(\La_2^\de;w_\de)\atop
			\dim E_k=k}\max_{0\neq v\in E_k}\f{\int_{\La_2^\de} w_\de |\pt_{x_2} v|^2\df x_2}{\int_{\La_2^\de} |v|^2\df x_2} 
		\end{split}
	\end{equation*}
	where $\Phi_{\de,i}^{(2)}\ (i\in\N^*)$ is defined in \eqref{09.18.1} and \eqref{09.18.2} and \eqref{09.18.3}. 
	
	\noindent
	{\it Step 2}. We prove that $
	\lambda_{\de_1,k}^{(2)}
	\leq
	\lambda_{\de_2,k}^{(2)}, 
	0\leq\de_1\leq\de_2<\frac18$. 
	
	Indeed, let $
	E_k^{\de_2}
	=
	\operatorname{Span}
	\{
	\Phi_{\de_2,1}^{(2)},\ldots,
	\Phi_{\de_2,k}^{(2)}
	\}$. 
	Since $
	\La_2^{\de_2}\s \La_2^{\de_1}$,  
	the zero extension 
	$$
	E_{\de_2}^{\de_1}:
	H_0^1(\La_2^{\de_2};w_{\de_2})
	\ra 
	H_0^1(\La_2^{\de_1};w_{\de_1})
	$$ 
	is well defined in Definition \ref{09.17.D2}. Moreover, $
	E_{\de_2}^{\de_1}E_k^{\de_2}
	\s  H_0^1(\La_2^{\de_1};w_{\de_1})$ with $
	\dim(E_{\de_2}^{\de_1}E_k^{\de_2})=k$. 
	For every $v\in E_k^{\de_2}$, the zero extension preserves both the
	$L^2$ norm and the weighted Dirichlet energy. Hence 
	$$
	\frac{
		a_{\de_1}^{(2)}(E_{\de_2}^{\de_1}v,E_{\de_2}^{\de_1}v)}
	{\|E_{\de_2}^{\de_1}v\|_{L^2(\La_2^{\de_1})}^2}
	=
	\frac{
		a_{\de_2}^{(2)}(v,v)}
	{\|v\|_{L^2(\La_2^{\de_2})}^2}.
	$$ 
	It follows from the min-max characterization that 
	$$
	\begin{aligned}
		\lambda_{\de_1,k}^{(2)}
		&\leq
		\max_{0\neq v\in E_k^{\de_2}}
		\frac{
			a_{\de_1}^{(2)}(E_{\de_2}^{\de_1}v,E_{\de_2}^{\de_1}v)}
		{\|E_{\de_2}^{\de_1}v\|_{L^2(\La_2^{\de_1})}^2} =
		\max_{0\neq v\in E_k^{\de_2}}
		\frac{
			a_{\de_2}^{(2)}(v,v)}
		{\|v\|_{L^2(\La_2^{\de_2})}^2}
		=
		\lambda_{\de_2,k}^{(2)}.
	\end{aligned}
	$$ 
	Therefore,
	\begin{equation}\label{09.18.4}
		\lambda_{\de_1,k}^{(2)}
		\leq
		\lambda_{\de_2,k}^{(2)},
		\mbox{ for all } k\in\N^* \mbox{ and all }
		0\leq\de_1\leq\de_2<\frac18.
	\end{equation}
	
	\noindent
	{\it Step 3}. We prove that $\la_{\de,k}^{(2)}$ is continuous at
	$\de=0$.
	
	From Lemma \ref{08.04.L11} and $
	r_\delta=E_\delta^0\left(A_\delta^{(2)}\right)^{-1}, 
	\de\in[0,\tfrac18)$ with $H=L^2(\La_2)$, 
	it is sufficient to prove that
	\begin{equation}\label{09.18.7}
		\|r_\de-r_0\|_{\mcL(L^2(\La_2))}
		\ra0 \mbox{ as }\de\ra0^+.
	\end{equation}
	i.e., eigenvalue convergence follows from operator norm convergence. 
	Indeed, since $r_\de$ and $r_0$ are compact, self-adjoint and
	nonnegative operators on $H=L^2(\La_2)$, their positive eigenvalues satisfy 
	$$
	\xi_k(r_\de)=\frac{1}{\la_{\de,k}^{(2)}},
	\quad
	\xi_k(r_0)=\frac{1}{\la_k^{(2)}}.
	$$ 
	Hence, by Lemma \ref{08.04.L11},
	\begin{equation*}
		\left|
		\frac{1}{\la_{\de,k}^{(2)}} -
		\frac{1}{\la_k^{(2)}}
		\right|
		\leq
		\|r_\de-r_0\|_{\mcL(H)}
		\ra0.
	\end{equation*}
	Since $\f{(1-\al)^2}{4}\leq \la_k^{(2)}\leq \la_{\de,k}^{(2)}\leq \la_{\f{1}{8},k}^{(2)}$ , we consequently obtain 
	$$
	\la_{\de,k}^{(2)}
	\ra
	\la_k^{(2)} \mbox{ as }\de\ra0^+.
	$$ 
	
	It remains to prove \eqref{09.18.7}.
	
	Suppose, by contradiction, that
	\eqref{09.18.7} does not hold. Then there exist $\e_0>0$, a sequence
	$\de_j\ra0^+$, and $f_j\in H$ with $\|f_j\|_H=1$ such that
	\begin{equation}\label{09.18.6}
		\|(r_{\de_j}-r_0)f_j\|_H\geq\e_0.
	\end{equation}
	By the uniform estimate in \eqref{09.18.9},  $
	\|r_{\de_j}f_j\|_{H_0^1(\La_2;w)}
	\leq C, 
	\|r_0f_j\|_{H_0^1(\La_2;w)}
	\leq C$. 
	Therefore, by the compact embedding $
	H_0^1(\La_2;w)\hookrightarrow H$, 
	after passing to a subsequence, there exist $u,v\in H$ such that 
	$$
	r_{\de_j}f_j\ra u,
	\mbox{ and }
	r_0f_j\ra v \mbox{ strongly in }H.
	$$ 
	Since ${f_j}$ is bounded in $H$, after passing to a further subsequence, 
	$$
	f_j\ra f \mbox{ weakly in }H
	$$ 
	for some $f\in H$.
	
	We claim that $
	u=v=r_0f$. 
	
	Indeed, by the self-adjointness of $r_0$, 
	$$
	(r_0f_j,g)_H=(f_j,r_0g)_H
	\ra(f,r_0g)_H
	=(r_0f,g)_H
	$$ 
	for every $g\in H$. Hence $
	r_0f_j\ra r_0f$ weakly in $H$. 
	Together with the strong convergence of $r_0f_j$ in $H$, this gives $
	v=r_0f$.  
	On the other hand, for every $g\in H$, $
	(r_{\de_j}f_j,g)_H
	=
	(f_j,r_{\de_j}g)_H$. 
	By Lemma \ref{09.17.L2}, $
	r_{\de_j}g\ra r_0g$ strongly in $H$, 
	while $f_j\ra f$ weakly in $H$. Thus 
	$$
	(f_j,r_{\de_j}g)_H
	\ra
	(f,r_0g)_H
	=
	(r_0f,g)_H.
	$$ 
	Consequently, 
	$$
	r_{\de_j}f_j\ra r_0f \mbox{ weakly in }H.
	$$ 
	Since $r_{\de_j}f_j\to u$ strongly in $H$, we obtain  $
	u=r_0f$.  
	Therefore, 
	$$
	\|(r_{\de_j}-r_0)f_j\|_H
	\ra0,
	$$ 
	which contradicts \eqref{09.18.6}. Hence \eqref{09.18.5} holds. We prove the lemma. 
\end{proof}

\begin{lemma}\label{09.18.L1}
	For each $k\in\N^*$, we have 
	\begin{equation*}
		E_\de^0\Phi_{\de,k}^{(2)}\ra \Phi_k ^{(2)} \mbox{ strongly in } H_0^1(\La_2;w) \mbox{ as }\de\ra 0^+. 
	\end{equation*}
\end{lemma}

\begin{proof}
	We prove this lemma by the following steps. 
	
	{\it Step 1}. We have known that the eigenvalue $\la_k^{(2)}$ is simple from \eqref{08.05.7}, since $\mcA^{(2)}$ is a Strum-Liouville operator. 
	
	{\it Step 2}. We prove $E_\de^0\Phi_{\de,k}^{(2)}\ra \Phi_k^{(2)}$ weakly in $H_0^1(\La_2;w)$ as $\de \ra 0^+$. 
	
	From \eqref{09.18.3} and Definition \ref{09.17.D2} and Lemma \ref{08.04.L10}, we obtain 
	\begin{equation*}
		\begin{split}
			\|E_\de^0\Phi_{\de,k}^{(2)}\|_{H_0^1(\La_2;w)}^2
			&=\|\Phi_{\de,k}^{(2)}\|_{H_0^1(\La_2^\de;w_\de)}^2=\lambda_{\delta,k}^{(2)}
\left\|\Phi_{\delta,k}^{(2)}\right\|_{L^2(\Lambda_\delta^s)}^2=\la_{\de,k}^{(2)}\leq \la_{\f{1}{8},k}^{(2)}, 
		\end{split}
	\end{equation*}
	then, up to a subsequence, and $u_{k}\in H_0^1(\La_2;w)$ such that 
	\begin{equation}\label{09.18.10}
		E_\de^0\Phi_{\de,k}^{(2)}\ra u_{k} \mbox{ weakly in } H_0^1(\La_2;w) \mbox{ and } E_\de^0\Phi_{\de,k}^{(2)}\ra u_k \mbox{ strongly in } L^2(\La_2) \mbox{ as } \de\ra 0^+.  
	\end{equation}
	For each $v\in C_0^\iy(\La_2)$, for all $\de\in (0,\dist(\supp v,\{0\}))$, we have $v\in C_0^\iy(\La_2^\de)$ and 
	\begin{equation*}
		\begin{split} 
		\int_{\La_2} w(\pt_{x_2}  E_\de^0\Phi_{\de,k}^{(2)})(\pt_{x_2}  \ol v)\df x_2
		&=\int_{\La_2^\de} w(\pt_{x_2}  \Phi_{\de,k}^{(2)})(\pt_{x_2}  \ol v)\df x_2\\
		&=\la_{\de,k}^{(2)}\int_{\La_2^\de} \Phi_{\de,k}^{(2)}\ol v\df x_2=\la_{\de,k}^{(2)}\int_{\La_2}(E_\de^0\Phi_{\de,k}^{(2)})\ol v\df x_2
		\end{split} 
	\end{equation*} 
	by Definition \ref{09.17.D2}. Combining this and \eqref{09.18.10} and Lemma \ref{08.04.L10}, we get
	\begin{equation*}
		\int_{\La_2} w(\pt_{x_2}u_k)(\pt_{x_2}\ol v)\df x_2=\la_k^{(2)}\int_{\La_2}u_k\ol v\df x_2, \mbox{ for all } v\in C_0^\iy(\La_2). 
	\end{equation*}
	This shows that $u_k$ is the $k$-th eigenfunction of $\mcA^{(2)}$ with respect to $\la_k^{(2)}$. Moreover, from \eqref{09.18.10} we obtain $\|u_k\|_{L^2(\La_2)}=1$ and hence $u_k=\Phi_k^{(2)}$ by Step 1. Otherwise, we replace $\Phi_k^{(2)}$ by $u_k$, which amounts to choosing the sign of the normalized eigenfunction $\Phi_k^{(2)}$ consistently.

	{\it Step 3}. We prove $\|E_\de^0\Phi_{\de,k}^{(2)}\|_{H_0^1(\La_2;w)}\ra \|\Phi_k^{(2)}\|_{H_0^1(\La_2;w)}$ as $\de\ra 0^+$. 
	
	Indeed, from Definition \ref{09.17.D2} and Lemma \ref{08.04.L10}, we have 
	\begin{equation*}
		\|E_\de^0\Phi_{\de,k}^{(2)}\|_{H_0^1(\La_2;w)}^2=a_\de^{(2)}(\Phi_{\de,k}^{(2)}, \Phi_{\de,k}^{(2)})=\la_{\de,k}^{(2)}\ra \la_k^{(2)} =a^{(2)}(\Phi_k^{(2)},\Phi_k^{(2)})=\|\Phi_k^{(2)}\|_{H_0^1(\La_2;w)}^2. 
	\end{equation*}

	{\it Step 4}. From Steps 2 and 3, together with the standard fact that if $u_n\ra u$ weakly in a Hilbert space $H$ and $\|u_n\|_H\to\|u\|_H$, then $u_n\to u$ strongly in $H$, we conclude the proof.
\end{proof}

\begin{lemma}\label{08.05.L2}
	For each $k\in\mathbb N^*$, we have
	\begin{equation}\label{09.18.11}
		\pt_{x_2}\Phi_{\de,k}^{(2)}(1) \ra \pt_{x_2} \Phi_k^{(2)}(1) \mbox{ as } \de\ra 0^+, 
	\end{equation}
	and 
	\begin{equation}\label{09.18.12}
		\left|\partial_{x_2}\Phi_{\delta,k}^{(2)}(1)\right|^2
		\le C\lambda_{\delta,k}^{(2)} \mbox{ for all } \de\in \left(0,\f{1}{8}\right)\mbox{ and all } k\in\N^*, 
	\end{equation}
	where $C>0$ depends only on $\alpha$.
\end{lemma}

\begin{proof}
	We prove this lemma in three steps.
	
	{\it Step 1}. We first show that $
	\Phi_{\delta,k}^{(2)}\in H^2(\frac14,1)$ 
	for all $\delta\in[0,\frac18)$ and all $k\in\mathbb N^*$.
	
	Since
	\begin{equation*}
		-x_2^\alpha\partial_{x_2x_2}\Phi_{\delta,k}^{(2)}
		-\alpha x_2^{\alpha-1}\partial_{x_2}\Phi_{\delta,k}^{(2)}
		=\mathcal A_\delta^{(2)}\Phi_{\delta,k}^{(2)}
		=\lambda_{\delta,k}^{(2)}\Phi_{\delta,k}^{(2)}
	\end{equation*}
	on $(\frac14,1)$, we have
	\begin{equation*}
		\partial_{x_2x_2}\Phi_{\delta,k}^{(2)}
		=
		-x_2^{-\alpha}\lambda_{\delta,k}^{(2)}
		\Phi_{\delta,k}^{(2)}
		-\alpha x_2^{-1}\partial_{x_2}\Phi_{\delta,k}^{(2)}.
	\end{equation*}
	Therefore, by \eqref{09.18.3} and Lemma \ref{09.18.L1},
	\begin{equation*}
		\begin{aligned}
			\|\partial_{x_2x_2}\Phi_{\delta,k}^{(2)}\|_{L^2(\frac14,1)}
			&\leq
			4^\alpha\lambda_{\delta,k}^{(2)}
			\|\Phi_{\delta,k}^{(2)}\|_{L^2(\Lambda_2)}
			+\alpha 4^{1+\frac\alpha2}
			\|x_2^{\frac\alpha2}
			\partial_{x_2}\Phi_{\delta,k}^{(2)}\|_{L^2(\Lambda_2)}
			\\
			&\leq
			4^\alpha\lambda_{\frac18,k}^{(2)}
			+\alpha 4^{1+\frac\alpha2}
			\bigl(\lambda_{\frac18,k}^{(2)}\bigr)^{\f{1}{2}} \leq C\lambda_{\f{1}{8},k}^{(2)},
		\end{aligned}
	\end{equation*}
	where $C>0$ depends only on $\alpha$. Here we have used the monotonicity $
	\lambda_{\delta,k}^{(2)}
	\leq\lambda_{\frac18,k}^{(2)}$   (see Lemma \ref{08.04.L10}) 
	and the uniform lower bound for $\lambda_{\delta,k}^{(2)}$. 
	Hence $
	\partial_{x_2x_2}\Phi_{\delta,k}^{(2)}
	\in L^2(\frac14,1)$, 
	and consequently $
	\Phi_{\delta,k}^{(2)}
	\in H^2(\frac14,1)$ 
	for all $\delta\in[0,\frac18)$ and all $k\in\mathbb N^*$.
	
	{\it Step 2}. We next prove \eqref{09.18.11}.
	
	By Lemmas \ref{08.04.L10} and \ref{09.18.L1},
	$$
	\lambda_{\delta,k}^{(2)}
	\to\lambda_k^{(2)}, \mbox{ and } 
	E_\delta^0\Phi_{\delta,k}^{(2)}
	\to\Phi_k^{(2)} \mbox{ strongly in }H_0^1(\Lambda_2;w).
	$$
	Using the equation satisfied by the eigenfunctions, we obtain
	\begin{equation*}
		\begin{aligned}
			&\|\partial_{x_2x_2}\Phi_{\delta,k}^{(2)}
			-\partial_{x_2x_2}\Phi_k^{(2)}
			\|_{L^2(\frac14,1)}
			\\
			&\leq
			4^\alpha
			\left|
			\lambda_{\delta,k}^{(2)}-\lambda_k^{(2)}
			\right|
			\|\Phi_{\delta,k}^{(2)}\|_{L^2(\Lambda_2)}  
			+4^\alpha\lambda_k^{(2)}
			\|\Phi_{\delta,k}^{(2)}-\Phi_k^{(2)}
			\|_{L^2(\frac14,1)}
			\\
			&\hspace{4.5mm}
			+\alpha 4^{1+\frac\alpha2}
			\|x_2^{\frac\alpha2}
			(\partial_{x_2}\Phi_{\delta,k}^{(2)}
			-\partial_{x_2}\Phi_k^{(2)})
			\|_{L^2(\frac14,1)}
			\\
			&\leq
			C\left|
			\lambda_{\delta,k}^{(2)}-\lambda_k^{(2)}
			\right|
			+C\lambda_k^{(2)}
			\|E_\delta^0\Phi_{\delta,k}^{(2)}
			-\Phi_k^{(2)}\|_{L^2(\Lambda_2)} 
			+C\|E_\delta^0\Phi_{\delta,k}^{(2)}
			-\Phi_k^{(2)}\|_{H_0^1(\Lambda_2;w))}
			\ra 0
		\end{aligned}
		\end{equation*}
		as $\delta\to0^+$. 
		Consequently,
		$$
		\Phi_{\delta,k}^{(2)}
		\to\Phi_k^{(2)} \mbox{ strongly in }H^2\left(\frac14,1\right).
		$$
		By the trace theorem,
		$$
		\partial_{x_2}\Phi_{\delta,k}^{(2)}(1)
		\to
		\partial_{x_2}\Phi_k^{(2)}(1) \mbox{ as } \de\ra 0^+, 
		$$
		which proves \eqref{09.18.11}.
		
		{\it Step 3}. We finally prove \eqref{09.18.12}.
		Since $A_\delta^{(2)}$ is a Sturm--Liouville operator, its eigenfunctions
$\Phi_{\delta,k}^{(2)}$ can be chosen to be real-valued.
		
		Let $\zeta\in C^\infty(\mathbb R)$, $0\leq\zeta\leq1$, satisfy
		\begin{equation*}
			\zeta=0 \mbox{ on }\left(-\infty,\frac14\right),
			\quad
			\zeta=1 \mbox{ on }\left(\frac12,+\infty\right),
			\quad
			|\zeta'|\leq8 \mbox{ on }\mathbb R.
		\end{equation*}
		Multiplying $
		\mathcal A_\delta^{(2)}
		\Phi_{\delta,k}^{(2)}
		=\lambda_{\delta,k}^{(2)}
		\Phi_{\delta,k}^{(2)}$ 
		by $\zeta\partial_{x_2}\Phi_{\delta,k}^{(2)}$ and integrating over $\Lambda_2^\delta$, we obtain
		\begin{equation*}
			\begin{aligned}
				\frac12
				\left(\partial_{x_2}\Phi_{\delta,k}^{(2)}(1)\right)^2
				&=
				\int_{\Lambda_2^\delta}
				x_2^\alpha\zeta'
				\left(\partial_{x_2}\Phi_{\delta,k}^{(2)}\right)^2\,dx_2 
				-\frac12
				\int_{\Lambda_2^\delta}
				\left(\partial_{x_2}\Phi_{\delta,k}^{(2)}\right)^2
				\partial_{x_2}(x_2^\alpha\zeta)\,dx_2 
				\\
				&\hspace{4.5mm}+\frac12\lambda_{\delta,k}^{(2)}
				\int_{\Lambda_2^\delta}
				\zeta'
				\left(\Phi_{\delta,k}^{(2)}\right)^2\,dx_2.
			\end{aligned}
		\end{equation*}
		Since $\zeta'$ is supported in $(\frac14,\frac12)$ and $
		\left|\partial_{x_2}(x_2^\alpha\zeta)\right|
		\leq C  x_2^\alpha \mbox{ on }(\tfrac14,1)$, 
		it follows that
		\begin{equation*}
			\begin{aligned}
				\frac12
				\left|\partial_{x_2}\Phi_{\delta,k}^{(2)}(1)\right|^2
				&\leq
				C\|\Phi_{\delta,k}^{(2)}
				\|_{H_0^1(\Lambda_2;w)}^2
				+C\lambda_{\delta,k}^{(2)} \leq C\lambda_{\delta,k}^{(2)}.
			\end{aligned}
		\end{equation*}
		Here, $C>0$ depends only on $\alpha$. This proves \eqref{09.18.12}, and completes the proof. 
	\end{proof}

\subsection{Approximation}

In this subsection, we regard the case $\delta=0$ as the original
degenerate problem \eqref{09.15.1}, while the cases
$\delta\in(0,\frac18)$ correspond to the uniformly elliptic shape
design problems \eqref{09.16.27}. Although the operators
$\mathcal A_\delta$ share the same differential expression, their
analytic properties are essentially different. Indeed, $\mathcal
A_0=\mcA $ is degenerate on the boundary $\{x_2=0\}$, whereas
$\mathcal A_\delta$ is uniformly elliptic for every $\delta>0$.
Consequently, the domains of these operators and the associated
regularity properties are different. 

Let $\vp^0\in C_0^\iy(\Om)$. Denote 
\begin{equation}\label{09.17.6}
	\de_0=\min\left\{\f{1}{8}, \f{1}{2}\dist(\supp \vp^0, \pt\Om) \right\}. 
\end{equation}
Then for each $\de\in (0,\de_0)$, $\vp^0\in C_0^\iy(\Om_\de)$. Note that $C_0^\iy(\Om)\s D(\mcA^m)$ for all $m\in \N$, then 
\begin{equation}\label{09.17.5}
	\begin{split}
		\vp^0=\sum_{n,k=1}^\iy \vp_{\de,nk}^0 \Phi_n^{(1)}\Phi_{\de,k}^{(2)}, \mbox{ with } \vp_{\de,nk}^0=(\vp^0, \Phi_n^{(1)}\Phi_{\de,k}^{(2)})_{L^2(\Om_\de)} \mbox{ for all } n,k\in\N^*. 
	\end{split}
\end{equation}
Hence, the solution of the Schr\"odinger equations \eqref{09.16.27} with respect to $(\vp_\de^0, 0)$ has the form 
\begin{equation}\label{09.16.29}
	\begin{split}
		\vp_\de(t)=\sum_{n,k=1}^\iy e^{\im \la_{nk}^\de t}\vp_{\de,nk}^0\Phi_n^{(1)}\Phi_{\de,k}^{(2)}, \mbox{ for all } \de\in [0,\de_0). 
	\end{split}
\end{equation}

For $\delta=0$, the above representation reduces to
$$
\vp(t)
=
\sum_{n,k=1}^{\infty}
e^{\im\lambda_{nk}t}
\vp_{nk}^{0}
\Phi_n^{(1)}\Phi_k^{(2)},
$$
because $
\lambda_{nk}^{0}=\lambda_{nk}, 
\vp_{0,nk}^{0}=\vp_{nk}^{0}$, and $\Phi_{0,k}^{(2)}=\Phi_k^{(2)}$. 

\begin{remark}\label{09.17.R2}
	Let $\delta\in[0,\delta_0)$. Define
	$$
	M_\delta
	:=
	\sum_{j=0}^4\sum_{n,k=1}^{\infty}
	(\lambda_{nk}^{\delta})^j
	|\vp_{\delta,nk}^0|^2
	=\sum_{j=0}^4
	\|\mathcal A_\delta^\f{j}{2}\vp^0\|_{L^2(\Omega_\delta)}^2 .
	$$
	Since $
	\operatorname{dist}
	(\operatorname{supp}\vp^0,\partial\Omega)>0$, 
	the support of $\vp^0$ is away from the degenerate boundary.
	Hence, on $\operatorname{supp}\vp^0$, the operators
	$\mathcal A_\delta$ are uniformly elliptic for
	$\delta\in[0,\delta_0)$, and all derivatives of the coefficients
	are uniformly bounded. Therefore,
	$$
	M_\delta
	=
	\sum_{j=0}^4\int_{\Omega_\delta}
	|\mathcal A_\delta^\f{j}{2}\vp^0|^2\,dx
	\leq
	C
	\sum_{j=0}^{4}
	\|D^j\vp^0\|_{L^\iy(\Omega)}^2 ,
	$$
	where $C$ depends only on $\alpha$ and $\delta_0$.
	Consequently,
	$$
	M:=\sup_{\delta\in[0,\delta_0)}M_\delta+1<+\infty .
	$$
\end{remark}

\begin{lemma}\label{09.19.L1}
	Let $\vp^0\in C_0^\iy(\Om)$ and   $\vp_\de\ (0\leq \de<\de_0)$ be the solution  defined in \eqref{09.16.11} and \eqref{09.16.28}, respectively. Then 
	\begin{equation*}
		\vp_\de \in C([0,T]; D(\mcA_\de)), \mbox{ and } \pt_t\vp_\de\in C([0,T];L^2(\Om_\de)). 
	\end{equation*}
\end{lemma}

\begin{proof}
	
	Since $\vp^0\in C_0^\infty(\Omega_\de)$ and 
	$C_0^\infty(\Omega_\de)\s  D(\mathcal A_\de^\theta)$
	for all $\theta\geq0$, we have
	$\vp_\de^0\in D(\mathcal A_\de)$. 
	We prove this lemma by the following steps.  
	
	{\it Step 1}.  We prove $\vp_\de\in L^2(0,T; D(\mcA_\de))$.

	From the spectral expansion \eqref{09.16.11} and \eqref{09.16.28}, 
	we obtain
	$$
	\begin{aligned}
		\|\mathcal A_\de\vp_\de(t)\|_{L^2(\Omega_\de)}^2
		&=
		\left\|
		\sum_{n,k=1}^\iy
		e^{\im\lambda_{nk}^\de t}
		\lambda_{nk}^\de \vp_{\de, nk}^0
		\Phi_n^{(1)}\Phi_{\de,k}^{(2)} 
		\right\|_{L^2(\Omega_\de)}^2=
		\sum_{n,k=1}^\iy
		(\lambda_{nk}^\de)^2|\vp_{\de, nk}^0|^2
		 \\
		&=
		\|\mathcal A_\de\vp^0\|_{L^2(\Omega)}^2\leq\la_1^{-2}\sum_{n,k=1}^\iy (\la_{nk}^\de)^4|\vp_{\de,nk}^0|^2 = \la_1^{-2} \|\mcA_\de^2\vp^0\|_{L^2(\Om_\de)}^2.
	\end{aligned}
	$$  
	This shows that $\mcA_\de\vp_\de\in L^2(Q_\de)$.

	{\it Step 2}. We prove $\vp_\de\in C([0,T]; D(\mcA_\de))$. 
	
	Let $\e\in (0,1)$ be arbitrary. Then there exists $m>0$ such that for all $n,k\geq m$ we have 
	\begin{equation*}
		\la_n^{(1)}\geq \max\{1, M^\f{1}{2}\e^{-1}\}, \quad \la_k^{(2)}\geq \max\{1, M^\f{1}{2}\e^{-1}\},  
	\end{equation*}
	where the positive constant $M$ is defined in Remark \ref{09.17.R2}. 
	From Lemma \ref{08.04.L10} we obtain 
	\begin{equation*}
		\la_{\de,k}^{(2)}\geq \la_k^{(2)}\geq  \max\{1,M\e^{-2}\}, \mbox{ for all } \de\in \left(0,\f{1}{8}\right) \mbox{ and all } k\in\N^*. 
	\end{equation*} 
	
	Now, denote 
	\begin{equation*}
		A_m=\{(n,k)\in\N^2\colon n\geq m+1, \mbox{ or } k\geq m+1\}. 
	\end{equation*}
	
	For the low-frequency part, by the orthonormality of 
	$\{\Phi_n^{(1)}\Phi_{\delta,k}^{(2)}\}_{n,k\in\mathbb N^*}$ in 
	$L^2(\Omega_\delta)$, we have
	\begin{equation*}
		\begin{split}
			&\left\|
			\sum_{n,k\leq m}
			\left(
			e^{\im\lambda_{nk}^{\delta}t_1}
			-
			e^{\im\lambda_{nk}^{\delta}t_2}
			\right)
			\lambda_{nk}^{\delta}
			\varphi_{\delta,nk}^{0}
			\Phi_n^{(1)}
			\Phi_{\delta,k}^{(2)}
			\right\|_{L^2(\Omega_\delta)}^2 =
			\sum_{n,k\leq m}
			\left|
			e^{\im\lambda_{nk}^{\delta}t_1}
			-
			e^{\im\lambda_{nk}^{\delta}t_2}
			\right|^2
			(\lambda_{nk}^{\delta})^2
			|\varphi_{\delta,nk}^{0}|^2 .
		\end{split}
	\end{equation*}
	Moreover, since $
		|
		e^{\im\lambda_{nk}^{\delta}t_1}
		-
		e^{\im\lambda_{nk}^{\delta}t_2}
		|
		\leq
		\lambda_{nk}^{\delta}|t_1-t_2|$, 
	we obtain
	\begin{equation*}
		\begin{split}
			&\left\|
			\sum_{n,k\leq m}
			\left(
			e^{\im\lambda_{nk}^{\delta}t_1}
			-
			e^{\im\lambda_{nk}^{\delta}t_2}
			\right)
			\lambda_{nk}^{\delta}
			\varphi_{\delta,nk}^{0}
			\Phi_n^{(1)}
			\Phi_{\delta,k}^{(2)}
			\right\|_{L^2(\Omega_\delta)}^2 \leq
			|t_1-t_2|^2
			\sum_{n,k\leq m}
			(\lambda_{nk}^{\delta})^4
			|\varphi_{\delta,nk}^{0}|^2 \leq
			M|t_1-t_2|^2 ,
		\end{split}
	\end{equation*}
	where we have used Remark \ref{09.17.R2}. Hence,
	\begin{equation*}
		\left\|
		\sum_{n,k\leq m}
		\left(
		e^{\im\lambda_{nk}^{\delta}t_1}
		-
		e^{\im\lambda_{nk}^{\delta}t_2}
		\right)
		\lambda_{nk}^{\delta}
		\varphi_{\delta,nk}^{0}
		\Phi_n^{(1)}
		\Phi_{\delta,k}^{(2)}
		\right\|_{L^2(\Omega_\delta)}
		\leq
		M^{\frac12}|t_1-t_2|.
	\end{equation*}
	
	The high-frequency part is estimated as
	\begin{equation*} 
	\begin{aligned}
		\left\|
		\sum_{(n,k)\in A_m}
		e^{\im\lambda_{nk}^{\delta}t}
		\lambda_{nk}^{\delta}
		\varphi_{\delta,nk}^0
		\Phi_n^{(1)}\Phi_{\delta,k}^{(2)}
		\right\|_{L^2(\Omega_\delta)}^2 
		&=
		\sum_{(n,k)\in A_m}
		(\lambda_{nk}^{\delta})^2
		|\varphi_{\delta,nk}^0|^2 \\
		&\leq
		\frac{\e^2}{M}
		\sum_{n,k}
		(\lambda_{nk}^{\delta})^4
		|\varphi_{\delta,nk}^0|^2
		\leq C\e^2, 
	\end{aligned}
	\end{equation*} 
	then,   
	\begin{equation*} 
		\begin{split}
			&\|\mcA_\de\vp_\de(t_1)-\mcA_\de\vp_\de(t_2)\|_{L^2(\Om_\de)}\\
			&\leq \left\|\sum_{n,k\leq m}\left(e^{\im \la_{nk}^\de t_1}-e^{\im \la_{nk}^\de t_2}\right)\la_{nk}^\de \vp_{\de,nk}^0 \Phi_n^{(1)}\Phi_{\de,k}^{(2)}\right\|_{L^2(\Om_\de)}\\
			&\hspace{4.5mm}+\left\|\sum_{n,k\in A_m} e^{\im \la_{nk}^\de t_1 }\la_{nk}^\de \vp_{\de,nk}^0 \Phi_n^{(1)}\Phi_{\de,k}^{(2)}\right\|_{L^2(\Om_\de)}+\left\|\sum_{n,k\in A_m} e^{\im \la_{nk}^\de t_2} \la_{nk}^\de \vp_{\de,nk}^0 \Phi_n^{(1)}\Phi_{\de,k}^{(2)}\right\|_{L^2(\Om_\de)}\\
			&\leq C|t_1-t_2|+C\e\leq C\e 
		\end{split}
	\end{equation*}
	for $|t_1-t_2|$ small enough, 
	where the positive constant $C$ depends only on $\al$ and $m$ and $M$.  This shows that  
	$$
	\vp_\de\in C([0,T];D(\mathcal A_\de)).
	$$
	
	{\it Step 3}. We prove $\pt_t\vp_\de\in C([0,T];L^2(\Om_\de))$. 
	
	Since $
	\partial_t\vp_\de 
	=
	\sum_{n,k=1}^\iy 
	\im\lambda_{nk}^\de 
	e^{\im\lambda_{nk}^\de t}
	\vp_{\de, nk}^0
	\Phi_n^{(1)}\Phi_{\de,k}^{(2)}$, 
	and therefore
	$$
	\|\partial_t\vp_\de (t)\|_{L^2(\Omega_\de )}^2
	=
	\sum_{n,k=1}^\iy
	(\lambda_{nk}^\de)^2|\vp_{\de, nk}^0|^2
	=
	\|\mathcal A_\de \vp^0\|_{L^2(\Omega_\de)}^2 \leq \la_1^{-2}\|\mcA_\de^2\vp^0\|_{L^2(\Om_\de)}^2.
	$$
	Consequently, by the same argument as Step 2 we obtain 
	$$
	\partial_t\vp_\de\in C([0,T];L^2(\Omega_\de)).
	$$
	This proves the lemma.
\end{proof}

\begin{lemma}\label{09.19.L2}
	Let $\vp^0\in C_0^\iy(\Om)$ and $\vp_\de\ (0\leq \de<\de_0)$ be the solution defined in \eqref{09.16.11} and \eqref{09.16.28}, respectively. Then 
	\begin{equation*}
		\f{\pt\vp_\de}{\pt \nu_{A_\de}}\in C([0,T]; L^2(\pt\Om-\ol{\Ga_{\de,2}^\de})). 
	\end{equation*}
\end{lemma}

\begin{proof}
	We prove this lemma by the following steps. 
	
	{\it Step 1}. We estimate on $\Ga_{\de,1}^{-1}$. 
	
	Indeed, from \eqref{09.16.11} and  \eqref{09.16.28} and $\f{\pt \vp_\de}{\pt \nu_{A_\de}}=\nabla\vp_\de\cdot A_\de\nu=-\f{\pt\vp_\de}{\pt x_1}$ on $\Ga_{\de,1}^{-1}$, and \begin{equation*}
		\pt_{x_1}\Phi_n^{(1)}=\f{n\pi}{2}\cos\left(\f{n\pi(x_1+1)}{2}\right)=\sqrt{\la_n^{(1)}}\cos\left(\f{n\pi(x_1+1)}{2}\right)
	\end{equation*} 
	by Remark \ref{08.05.R1},  and $\la_{nk}^\de=\la_n^{(1)}+\la_k^{(2)}\geq \la_n^{(1)}$, we obtain 
	\begin{equation*}
		\begin{split}
			&\left\|\f{\pt\vp_\de}{\pt \nu_{A_\de}}(t)\right\|_{L^2(\Ga_{\de,1}^{-1})}^2\\
			&=\left\|\sum_{n,k=1}^\iy e^{\im \la_{nk}^\de t} \vp_{\de,nk}^0 \left[\pt_{x_1}\Phi_n^{(1)}(-1)\right] \Phi_{\de,k}^{(2)}\right\|_{L^2(\Ga_{\de,1}^{-1})}^2	=\left\|\sum_{n,k=1}^\iy e^{\im \la_{nk}^\de t} \vp_{\de,nk}^0 \sqrt{\la_n^{(1)}} \Phi_{\de,k}^{(2)}\right\|_{L^2(\Ga_{\de,1}^{-1})}^2\\
			&= \sum_{k=1}^\iy\left| \sum_{n=1}^\iy e^{\im \la_{nk}^\de t} \vp_{\de,nk}^0 \sqrt{\la_n^{(1)}}\right|^2=\sum_{k=1}^\iy \left|\sum_{n=1}^\iy e^{\im\la_{nk}^\de t} \vp_{\de,nk}^0  \la_{nk}^\de(\la_{nk}^\de)^{-1} \sqrt{\la_n^{(1)}}\right|^2\\
			&\leq \sum_{k=1}^\iy \left(\sum_{n=1}^\iy (\la_{nk}^\de)^2|\vp_{\de,nk}^0|^2\right)\left(\sum_{n=1}^\iy \f{\la_n^{(1)}}{(\la_{nk}^\de)^2}\right)\leq \left(\sum_{n=1}^\iy \f{1}{\la_n^{(1)}}\right)\sum_{n,k=1}^\iy (\la_{nk}^\de)^2|\vp_{\de,nk}^0|^2\\
			&\leq \la_1^{-2} C\|\mcA_\de\vp^0\|_{L^2(\Om_\de)}^2, 
		\end{split}
	\end{equation*}
	where the positive constant $C$ depends only on $\al$.  
	
	Now, by the same argument as Step 2 in the proof of Lemma \ref{09.19.L1}, we obtain $\f{\pt\vp_\de}{\pt\nu_{A_\de}}\in C([0,T]; L^2(\Ga_{\de,1}^{-1}))$. 
	
	{\it Step 2}. We estimate on $\Gamma_{\delta,1}^{1}$.
	
	The estimate on $\Gamma_{\delta,1}^{1}$ follows from the same
	calculation as that on $\Gamma_{\delta,1}^{-1}$.
	
	{\it Step 3}. We estimate on $\Gamma_{\delta,2}^{1}$.
	
	Indeed, from \eqref{09.16.11} and \eqref{09.16.28}, and $
	\frac{\partial\varphi_\delta}{\partial\nu_{A_\delta}}
	=
	\nabla\varphi_\delta\cdot A_\delta\nu
	=
	\frac{\partial\varphi_\delta}{\partial x_2}$ on $\Gamma_{\delta,2}^{1}$,  and Remark \ref{08.05.R1} and $\la_{\de, nk}^\de=\la_n^{(1)}+\la_{\de,k}^{(2)}\geq \la_{\de,k}^{(2)}\geq \la_k^{(2)}$ (see Lemma \ref{08.04.L10}), and  \eqref{09.18.12},  we obtain
	\begin{equation*}
		\begin{aligned}
			&\left\|
			\frac{\partial\varphi_\delta}{\partial\nu_{A_\delta}}(t)
			\right\|_{L^2(\Gamma_{\delta,2}^{1})}^2 \\
			&=
			\left\|
			\sum_{n,k=1}^{\infty}
			e^{\im\lambda_{nk}^{\delta}t}
			\varphi_{\delta,nk}^{0}
			\Phi_n^{(1)}
			\partial_{x_2}\Phi_{\delta,k}^{(2)}(1)
			\right\|_{L^2(\Gamma_{\delta,2}^{1})}^2=\sum_{k=1}^{\infty}
			\left|
			\sum_{n=1}^{\infty}
			e^{\im\lambda_{nk}^{\delta}t}
			\varphi_{\delta,nk}^{0}
			\partial_{x_2}\Phi_{\delta,k}^{(2)}(1)
			\right|^2
			\\
			&=
			\sum_{k=1}^{\infty}
			\left|
			\sum_{n=1}^{\infty}
			e^{\im\lambda_{nk}^{\delta}t}
			\lambda_{nk}^{\delta}\varphi_{\delta,nk}^{0}
			(\lambda_{nk}^{\delta})^{-1}
			\partial_{x_2}\Phi_{\delta,k}^{(2)}(1)
			\right|^2 \leq
			\sum_{k=1}^{\infty}
			\left(
			\sum_{n=1}^{\infty}
			(\lambda_{nk}^{\delta})^2
			|\varphi_{\delta,nk}^{0}|^2
			\right)
			\left(
			\sum_{n=1}^{\infty}
			\frac{|\partial_{x_2}\Phi_{\delta,k}^{(2)}(1)|^2}
			{(\lambda_{nk}^{\delta})^2}
			\right)\\ 
			&\leq
			C
			\sum_{n,k=1}^{\infty}
			(\lambda_{nk}^{\delta})^2
			|\varphi_{\delta,nk}^{0}|^2 \leq
			C\|\mathcal A_\delta\varphi^0\|_{L^2(\Omega_\delta)}^2, 
		\end{aligned}
	\end{equation*}
	where the positive constants $C$ depend only on $\al$. 
	
	Finally, by the same argument as Step 2 in the proof of Lemma \ref{09.19.L1}, we obtain $\f{\pt\vp_\de}{\pt\nu_{A_\de}}\in C([0,T]; L^2(\Ga_{\de,2}^{1}))$. We prove the lemma. 
\end{proof}

\begin{lemma}\label{09.20.L1}
	Let $\vp^0\in C_0^\iy(\Om)$ and 
	$\vp_\de\ (0\leq \de<\de_0)$ be the solution defined in 
	\eqref{09.16.11} and \eqref{09.16.28}, respectively. Then 
	\begin{equation*}
		\|\vp_\de(t)\|_{L^2(\Om_\de)}
		=
		\|\vp^0\|_{L^2(\Om_\de)},
	\end{equation*}
	and
	\begin{equation*}
		\|\vp_\de(t)\|_{H_0^1(\Om_\de;w_\de)}
		=
		\|\vp^0\|_{H_0^1(\Om_\de;w_\de)} .
	\end{equation*}
\end{lemma}

\begin{proof}
	From \eqref{09.16.11} and \eqref{09.16.28}, we have
	\begin{equation*}
		\vp_\de(t)
		=
		\sum_{n,k=1}^{\iy}
		e^{\im\la_{nk}^\de t}
		\vp_{\de,nk}^0
		\Phi_n^{(1)}
		\Phi_{\de,k}^{(2)} .
	\end{equation*}
	Since 
	$\{\Phi_n^{(1)}\Phi_{\de,k}^{(2)}\}_{n,k\in\mathbb N^*}$
	is an orthonormal basis of $L^2(\Omega_\de)$, we obtain
	\begin{equation*}
		\begin{aligned}
			\|\vp_\de(t)\|_{L^2(\Om_\de)}^2
			&=
			\sum_{n,k=1}^{\iy}
			\left|
			e^{\im\la_{nk}^\de t}
			\vp_{\de,nk}^0
			\right|^2 =
			\sum_{n,k=1}^{\iy}
			|\vp_{\de,nk}^0|^2
			=
			\|\vp^0\|_{L^2(\Om_\de)}^2 .
		\end{aligned}
	\end{equation*}
	
	Moreover, by the spectral characterization of 
	$H_0^1(\Omega_\de;w_\de)=D(\mathcal A_\de^{\f{1}{2}})$ by \eqref{09.15.8} and Remark \ref{09.17.R1}, we have
	\begin{equation*}
		\begin{aligned}
			\|\vp_\de(t)\|_{H_0^1(\Om_\de;w_\de)}^2
			&=
			\sum_{n,k=1}^{\iy}
			\la_{nk}^\de
			\left|
			e^{\im\la_{nk}^\de t}
			\vp_{\de,nk}^0
			\right|^2 =
			\sum_{n,k=1}^{\iy}
			\la_{nk}^\de
			|\vp_{\de,nk}^0|^2=
			\|\vp^0\|_{H_0^1(\Om_\de;w_\de)}^2 .
		\end{aligned}
	\end{equation*}
	Hence,
	\begin{equation*}
		\|\vp_\de(t)\|_{H_0^1(\Om_\de;w_\de)}
		=
		\|\vp^0\|_{H_0^1(\Om_\de;w_\de)} .
	\end{equation*}
	This proves the lemma.
\end{proof}

\begin{definition}\label{09.17.D3}
	Let $\vp_\de\in L^2(0,T; H_0^1(\Om_\de;w_\de))$. Define
	\begin{equation*}
		E_\de^0 \vp_\de(x,t)=
		\begin{cases}
			\vp_\de(x,t), &(x,t)\in Q_\de, \\
			0, &(x,t)\in Q-Q_\de. 
		\end{cases}
	\end{equation*}
	Then $E_\de^0 \vp_\de\in L^2(0,T; H_0^1(\Om;w))$ and 
	\begin{equation*}
		\|E_\de^0 \vp_\de\|_{L^2(0,T; H_0^1(\Om;w))}=\|\vp_\de\|_{L^2(0,T; H_0^1(\Om_\de;w_\de))}. 
	\end{equation*}

	Let $\vp_\de\in H^1(0,T; L^2(\Om_\de))$. Define $E_\de^0 \vp_\de$ as above, then $E_\de^0\vp_\de\in H^1(0,T; L^2(\Om))$ and 
	\begin{equation*}
		\pt_t E_\de^0\vp_\de=E_\de^0 \pt_t\vp_\de, \mbox{ and } \|\pt_tE_\de^0\vp_\de\|_{L^2(Q)}=\|\pt_t\vp_\de\|_{L^2(Q_\de)}. 
	\end{equation*}

\end{definition}

\begin{theorem}\label{09.17.T2}
	Let $\vp^0\in C_0^\infty(\Omega)$ and let $\delta_0>0$ be
	defined in \eqref{09.17.6}. Let $\vp$ and $\vp_\delta\ (0<\de<\de_0)$
	be the weak solutions of \eqref{09.15.1} and
	\eqref{09.16.27}, respectively, with initial data $\vp^0$ and $f=f_\de=0$.
	Then, up to a subsequence as $\delta\to0^+$,
	\begin{equation}\label{09.17.7}
		E_\delta^0\vp_\delta
		\to\vp \mbox{ strongly in }
		L^2(0,T; H_0^1(\Omega;w)),
	\end{equation}
	and
	\begin{equation}\label{09.17.8}
		\frac{\partial E_\delta^0\vp_\delta}{\partial\nu_A}
		\to
		\frac{\partial\vp}{\partial\nu_A} \mbox{ strongly in }
		L^2((\partial\Omega-\Gamma_2^0)\times(0,T)).
	\end{equation}
\end{theorem}

\begin{proof}
	We prove this theorem by the following steps. 
	
	{\it Step 1.}
	By the spectral decomposition established in \eqref{09.16.29}, the
	solution $\vp_\delta$ for $\delta\in[0,\delta_0)$ admits the
	representation
	$$
	\vp_\delta(t)
	=
	\sum_{n,k=1}^{\infty}
	e^{\im \lambda_{nk}^{\delta}t}
	\vp_{\delta,nk}^0
	\Phi_n^{(1)}\Phi_{\delta,k}^{(2)} .
	$$
	Note that $\Phi_n^{(1)}$ is defined in Remark \ref{08.05.R1}. 
	
	{\it Step 2}. We prove \eqref{09.17.7}. 
	
	Let $\e>0$. Then there exists $m>0$ such that for all $n,k\geq m$ we have 
	\begin{equation*}
		\la_n^{(1)}\geq \max\{1, M^\f{1}{2}\e^{-1}\}, \quad \la_k^{(2)}\geq \max\{1, M^\f{1}{2}\e^{-1}\},  \quad \sum_{n=m+1}^\iy \f{1}{\la_n^{(1)}}<\f{\e}{M}, \quad \sum_{k=m+1}^\iy \f{1}{\la_k^{(2
				)}}<\f{\e}{M}, 
	\end{equation*}
	where the positive constant $M$ is defined in Remark \ref{09.17.R2}. 
	From Lemma \ref{08.04.L10} we obtain 
	\begin{equation*}
		\la_{\f{1}{8},k}^{(2)}\geq \la_{\de,k}^{(2)}\geq \la_k^{(2)}\geq  \max\{1,M^\f{1}{2}\e^{-1}\}, \mbox{ for all } \de\in \left(0,\f{1}{8}\right) \mbox{ and all } k\in\N^*. 
	\end{equation*}
	
	Now, denote 
	\begin{equation*}
		A_m=\{(n,k)\in\N^2\colon n\geq m+1, \mbox{ or } k\geq m+1\}. 
	\end{equation*}
	Note that $m$ is independent of $\de\in [0,\f{1}{8})$. 
	Then, for all $\de\in [0,\f{1}{8})$, from \eqref{09.16.29}, we obtain 
	\begin{equation*}
		\begin{split}
			B_1^\de(m) 
			&=\left\|\sum_{n,k\in A_m} e^{\im \la_{nk}^\de t} \vp_{\de,nk}^0 \Phi_n^{(1)}\Phi_{\de,k}^{(2)}\right\|_{H_0^1(\Om_\de;w_\de)}=\left(\sum_{n,k\in A_m}\la_{nk}^\de |\vp_{\de,nk}^0|^2\right)^\f{1}{2}\\
			&\leq  \left(\sum_{n,k\in A_m}(\la_{nk}^{\de})^{-2}(\la_{nk}^\de)^4|\vp_{\de,nk}^0|^2\right)^\f{1}{2}\leq \left(\f{\e^2}{M}\sum_{n,k\in A_m}(\la_{nk}^\de)^4|\vp_{\de,nk}^0|^2\right)^\f{1}{2} <\e.  
		\end{split}
	\end{equation*}
	Note that 
	\begin{equation*}
		\begin{split}
			&\left\|E_\de^0\vp_\de(t)-\vp(t)\right\|_{H_0^1(\Om;w)}\\
			&\leq \left\|\sum_{n,k\leq m}e^{\im \la_{nk}^\de t} \vp_{\de,nk}^0 \Phi_n^{(1)}E_\de^0 \Phi_{\de,k}^{(2)}-\sum_{n,k\leq m}e^{\im \la_{nk}t} \vp_{nk}^0 \Phi_n^{(1)} \Phi_{k}^{(2)}\right\|_{H_0^1(\Om;w)}\\
			&\hspace{4.5mm}+\left\|\sum_{n,k\in A_m} e^{\im \la_{nk}^\de t} \vp_{\de,nk}^0 \Phi_n^{(1)}E_\de^0 \Phi_{\de,k}^{(2)}\right\|_{H_0^1(\Om;w)}+\left\|\sum_{n,k\in A_m}e^{\im \la_{nk}t} \vp_{nk}^0 \Phi_n^{(1)} \Phi_{k}^{(2)}\right\|_{H_0^1(\Om;w)}\\
			&=: B_0(m)+B_1^\de(m)+B_1^0(m). 
		\end{split}
	\end{equation*}
	Since  
	\begin{equation*}
		\begin{split}
			&\sum_{n,k\leq m}e^{\im \la_{nk}^\de t} \vp_{\de,nk}^0 \Phi_n^{(1)}E_\de^0 \Phi_{\de,k}^{(2)}-\sum_{n,k\leq m}e^{\im \la_{nk}t} \vp_{nk}^0 \Phi_n^{(1)}E_\de^0 \Phi_{k}^{(2)}\\
			&=\sum_{n,k\leq m} \left(e^{\im \la_{nk}^\de t}-e^{\im \la_{nk}t}\right) \vp_{\de,nk}^0 \Phi_n^{(1)}E_\de^0\Phi_{\de,k}^{(2)}+\sum_{n,k\leq m}e^{\im \la_{nk}t} (\vp_{\de,nk}^0-\vp_{nk}^0)\Phi_n^{(1)}E_\de^0\Phi_{\de,k}^{(2)}\\
			&\hspace{4.5mm}+\sum_{n,k\leq m}e^{\im \la_{nk}t}\vp_{nk}^0 \Phi_n^{(1)} (E_\de^0\Phi_{\de,k}^{(2)}-\Phi_k^{(2)}), 
		\end{split}
	\end{equation*}
	and from Lemma \ref{08.04.L10} and $|e^{\im \la_{nk}^\de t}-e^{\im \la_{nk}t}|\leq T|\la_{nk}^\de-\la_{nk}|$ for all $t\in [0,T]$ and $\la_{nk}^\de\leq \la_{nk}^\f{1}{8}$ for all $n,k\in \N^*$ we obtain 
	\begin{equation*}
		\begin{split}
			\left\|\sum_{n,k\leq m} \left(e^{\im \la_{nk}^\de t}-e^{\im \la_{nk}t}\right) \vp_{\de,nk}^0 \Phi_n^{(1)}E_\de^0\Phi_{\de,k}^{(2)}\right\|_{H_0^1(\Om;w)}^2
			&\leq T^2\sum_{n,k\leq m}|\la_{nk}^\de-\la_{nk}|^2\la_{nk}^\de|\vp_{\de,nk}^0|^2\leq MT^2\e^2
		\end{split}
	\end{equation*}
	for $\de>0$ small enough; from $|\vp_{\de,nk}^0-\vp_{nk}^0|=(\vp^0,\Phi_n^{(1)}\Phi_{\de,k}^{(2)})_{L^2(\Om)}-(\vp^0, \Phi_n^{(1)}\Phi_k^{(2)})$ and Lemma \ref{09.18.L1} and $\la_{nk}^\de\leq \la_{nk}^\f{1}{8}$ for all $n,k\in \N^*$ we obtain 
	\begin{equation*}
		\begin{split}
			&\left\|\sum_{n,k\leq m}e^{\im \la_{nk}t} (\vp_{\de,nk}^0-\vp_{nk}^0)\Phi_n^{(1)}E_\de^0\Phi_{\de,k}^{(2)}\right\|_{H_0^1(\Om;w)}^2\leq \sum_{n,k\leq m} |\vp_{\de,nk}^0-\vp_{nk}^0|^2 \la_{nk}^\de <\e 
		\end{split}
	\end{equation*}
	for $\de>0$ small enough; from Lemma \ref{09.18.L1} and $\la_{nk}^\de\leq \la_{nk}^\f{1}{8}$ for all $n,k\in \N^*$ we obtain 
	\begin{equation*}
		\begin{split}
			&\left\|\sum_{n,k\leq m}e^{\im \la_{nk}t}\vp_{nk}^0 \Phi_n^{(1)} (E_\de^0\Phi_{\de,k}^{(2)}-\Phi_k^{(2)})\right\|_{H_0^1(\Om;w)}^2\leq \sum_{n,k\leq m} |\vp_{nk}^0|^2 \la_n^{(1)} \|E_\de^0\Phi_{\de,k}^{(2)}-\Phi_k^{(2)}\|_{H_0^1(\La_2;w)}^2<\e 
		\end{split}
	\end{equation*}
	for $\de>0$ small enough, hence $B_0(m)\ra 0$ as $\de\ra 0$. Therefore, for a.e.~$t\in [0,T]$, we have  
	\begin{equation*}
		\left\|E_\de^0\vp_\de(t)-\vp(t)\right\|_{H_0^1(\Om;w)}\ra 0 \mbox{ as } \de\ra 0. 
	\end{equation*}
	Which shows  \eqref{09.17.7}. 
	
	{\it Step 3}. We prove 
	\begin{equation}\label{09.20.6}
		\begin{split} 
			\f{\pt E_\de^0\vp_\de}{\pt \nu_{A}}\ra \f{\pt \vp}{\pt \nu_A} \mbox{ strongly in } L^2(\Ga_1^{-1}\ts (0,T)). 
		\end{split}
	\end{equation}
	
	Let $\e>0, m\in\N^*$ and $A_m$ be defined in Step 2. Then, from $\pt_{x_1}\Phi_n^{(1)}(-1)=\sqrt{\la_n^{(1)}}$ (see Remark \ref{08.05.R1}) and 
	\begin{equation*}
		\begin{split}
			\f{\pt E_\de^0\vp_\de}{\pt \nu_A}=\nabla E_\de^0\vp_\de\cdot A_\de\nu=-[\pt_{x_1}E_\de^0\vp_\de](-1,\cdot) \mbox{  on } \Ga_1^{-1}\ts (0,T), 
		\end{split}
	\end{equation*}
	we obtain for all $\de\in [0,\de_0)$, 
	\begin{equation*}
		\begin{split}
			C_1^\de(m)
			&=\left\|\sum_{n,k\in A_m} e^{\im \la_{nk}^\de t} \vp_{\de,nk}^0 [-\pt_{x_1}\Phi_n^{(1)}](1)E_\de^0\Phi_{\de,k}^{(2)}\right\|_{L^2(\Ga_1^{-1})}\\
			&=\left\|\sum_{n,k\in A_m}e^{\im \la_{nk}^\de t}\vp_{\de,nk}^0 \sqrt{\la_n^{(1)}} E_\de^0\Phi_{\de,k}^{(2)} \right\|_{L^2(\Ga_1^{-1})}\\
			&=\left(\sum_{k=m+1}^\iy\left|\sum_{n=1}^\iy e^{\im \la_{nk}^\de t} \vp_{\de,nk}^0 \sqrt{\la_n^{(1)}} \right|^2+\sum_{k=1}^m\left|\sum_{n=m+1}^\iy e^{\im \la_{nk}^\de t} \vp_{\de,nk}^0 \sqrt{\la_{n}^{(1)}}\right|^2 \right)^\f{1}{2}. 
		\end{split}
	\end{equation*}
	Note that from $\la_{nk}^\de =\la_{n}^{(1)}+\la_{\de,k}^{(2)}\geq   \la_n^{(1)}$ we obtain 
	\begin{equation*}
		\begin{split}
			&\sum_{k=m+1}^\iy\left|\sum_{n=1}^\iy e^{\im \la_{nk}^\de t} \vp_{\de,nk}^0 \sqrt{\la_n^{(1)}} \right|^2\\
			&= \sum_{k=m+1}^\iy\left|\sum_{n=1}^\iy e^{\im \la_{nk}^\de t}\vp_{\de, nk}^0\la_{nk}^\de (\la_{nk}^\de)^{-1} \sqrt{\la_n^{(1)}}\right|^2 \leq \sum_{k=m+1}^\iy \left(\sum_{n=1}^\iy (\la_{nk}^\de)^2|\vp_{\de,nk}^0|^2\right)\left(\sum_{n=1}^\iy \f{\la_n^{(1)}}{(\la_{nk}^\de)^2}\right)\\
			&\leq \sum_{k=m+1}^\iy \left(\sum_{n=1}^\iy\f{1}{\la_n^{(1)}+\la_{\de,k}^{(2)}} (\la_{nk}^\de)^3|\vp_{nk}^0|^2\right) \left(\sum_{n=1}^\iy \f{1}{\la_n^{(1)}}\right) \leq C \f{\e}{M^\f{1}{2}} \sum_{n,k=1}^\iy (\la_{nk}^\de)^3|\vp_{nk}^0|^2<C\e, 
		\end{split}
	\end{equation*}
	and 
	\begin{equation*}
		\begin{split}
			&\sum_{k=1}^m\left|\sum_{n=m+1}^\iy e^{\im \la_{nk}^\de t} \vp_{\de,nk}^0 \sqrt{\la_{n}^{(1)}}\right|^2\\
			&=\sum_{k=1}^m \left|\sum_{n=m+1}^\iy e^{\im\la_{nk}^\de t} \vp_{\de,nk}^0 \la_{nk}^\de (\la_{nk}^{\de})^{-1}\sqrt{\la_n^{(1)}}\right|^2\leq \sum_{k=1}^m \left(\sum_{n=m+1}^\iy (\la_{nk}^\de)^2|\vp_{\de,nk}^0|^2 \right) \left(\sum_{n=m+1}^\iy\f{\la_n^{(1)}}{(\la_{nk}^\de)^2}\right) \\
			&\leq \left(\sum_{n=m+1}^\iy \f{1}{\la_n^{(1)}}\right)\left(\sum_{n,k=1}^\iy(\la_{nk}^\de)^2|\vp_{\de, nk}^0|^2\right)<\e. 
		\end{split}
	\end{equation*}
	Hence 
	\begin{equation*}
		C_1^\de(m)<C\e, \mbox{ for all }\de\in [0,\de_0). 
	\end{equation*}
	Combining this and 
	\begin{equation*}
		\begin{split}
			&\left\|\f{\pt E_\de^0\vp_\de}{\pt \nu_{A}}(t)- \f{\pt \vp}{\pt \nu_A}(t)\right\|_{L^2(\Ga_1^{-1})}^2\\
			&=\left\|\sum_{n,k=1}^m e^{\im \la_{nk}^\de t} \vp_{\de,nk}^0 [-\pt_{x_1}\Phi_n^{(1)}](1)E_\de^0\Phi_{\de,k}^{(2)}-\sum_{n,k=1}^m e^{\im \la_{nk} t}\vp_{nk}^0 [-\pt_{x_1}\Phi_n^{(1)}](1) \Phi_k^{(2)} \right\|_{L^2(\Ga_1^{-1})}^2\\
			&\hspace{4.5mm}+\left\|\sum_{n,k\in A_m} e^{\im \la_{nk}^\de t} \vp_{\de,nk}^0 [-\pt_{x_1}\Phi_n^{(1)}](1)E_\de^0\Phi_{\de,k}^{(2)}\right\|_{L^2(\Ga_1^{-1})}\\
			&\hspace{4.5mm}+\left\|\sum_{n,k\in A_m} e^{\im \la_{nk}^\de t} \vp_{\de,nk}^0 [-\pt_{x_1}\Phi_n^{(1)}](1)\Phi_{k}^{(2)}\right\|_{L^2(\Ga_1^{-1})}, 
		\end{split}
	\end{equation*}
	and 
	\begin{equation*}
		\begin{split}
			&\sum_{n,k=1}^m e^{\im \la_{nk}^\de t} \vp_{\de,nk}^0 [-\pt_{x_1}\Phi_n^{(1)}](1)E_\de^0\Phi_{\de,k}^{(2)}-\sum_{n,k=1}^m e^{\im \la_{nk} t}\vp_{nk}^0 [-\pt_{x_1}\Phi_n^{(1)}](1) \Phi_k^{(2)}\\
			&=\sum_{n,k=1}^m \left(e^{\im \la_{nk}^\de t}-e^{\im\la_{nk}t}\right)\vp_{\de,nk}^0 \sqrt{\la_n^{(1)}}E_\de^0\Phi_{\de,k}^{(2)}+\sum_{n,k=1}^m e^{\im \la_{nk}t}(\vp_{\de,nk}^0-\vp_{nk}^0)\sqrt{\la_n^{(1)}}E_\de^0\Phi_{\de,k}^{(2)}\\
			&\hspace{4.5mm}+\sum_{n,k=1}^m e^{\im\la_{nk}t}\vp_{nk}^0 \sqrt{\la_n^{(1)}}(E_\de^0\Phi_{\de,k}^{(2)}-\Phi_k^{(2)}),
		\end{split}
	\end{equation*}
	and 
	\begin{equation*}
		|e^{\im\la_{nk}^\de t}-e^{\im\la_{nk} t}|\leq T|\la_{nk}^\de-\la_{nk}| \mbox{ for all }t\in [0,T], \  \vp_{\de,nk}^0-\vp_{nk}^0=(\vp^0, \Phi_n^{(1)} (E_\de^0\Phi_{\de,k}^{(2)}-\Phi_k^{(2)}))_{L^2(\Om)},
	\end{equation*}
	and Lemmas \ref{08.04.L10} and \ref{09.18.L1}, we obtain 
	\begin{equation*}
		\left\|\f{\pt E_\de^0\vp_\de}{\pt \nu_{A}}(t)- \f{\pt \vp}{\pt \nu_A}(t)\right\|_{L^2(\Ga_1^{-1})}\ra 0 \mbox{ uniformly for }t\in [0,T] \mbox{ as } \de\ra 0^+. 
	\end{equation*}
	This proves \eqref{09.20.6} 
	
	{\it Step 4}. 
	By the same argument as in Step 3,  
	we obtain
	\begin{equation*}
		\frac{\partial E_\delta^0\vp_\delta}{\partial \nu_A}
		\to
		\frac{\partial \vp}{\partial \nu_A}
		\quad\mbox{strongly in }
		L^2(\Gamma_1^{1}\times(0,T)).
	\end{equation*}
	
	{\it Step 5}. 
	We prove 
	\begin{equation}\label{09.20.7}
		\begin{split} 
			\f{\pt E_\de^0\vp_\de}{\pt \nu_{A}}\ra \f{\pt \vp}{\pt \nu_A} \mbox{ strongly in } L^2(\Ga_2^{1}\ts (0,T)). 
		\end{split}
	\end{equation}
	
	Since 
	\begin{equation*}
		\f{\pt E_\de^0\vp_\de}{\pt \nu_A}=\nabla E_\de^0\vp_\de\cdot A_\de\nu =\pt_{x_2}\vp_\de(\cdot, 1) \mbox{ on }\Ga_2^1, 
	\end{equation*}
	we obtain 
	\begin{equation*}
		\begin{split}
			&\left\|\f{\pt E_\de^0\vp_\de}{\pt \nu_{A}}(t)- \f{\pt \vp}{\pt \nu_A}(t)\right\|_{L^2(\Ga_2^{1})}\\
			&=\left\|\sum_{n,k=1}^\iy  e^{\im \la_{nk}^\de t} \vp_{\de,nk}^0 \Phi_n^{(1)}[\pt_{x_2} E_\de^0\Phi_{\de,k}^{(2)}](1)-\sum_{n,k=1}^\iy e^{\im \la_{nk} t}\vp_{nk}^0 \Phi_n^{(1)}  [\pt_{x_2} \Phi_k^{(2)}](1) \right\|_{L^2(\Ga_2^{1})}. 
		\end{split}
	\end{equation*}
	Similar to Step 3, but exchanging the roles of
	$\lambda_n^{(1)}$ and $\lambda_{\delta,k}^{(2)}$, and using Lemma
	\ref{08.05.L2}, which gives the estimate of the boundary traces
	$|\partial_{x_2}\Phi_{\delta,k}^{(2)}(1)|$, we obtain
	$$
	\frac{\partial E_\delta^0\vp_\delta}{\partial\nu_A}
	\to
	\frac{\partial\vp}{\partial\nu_A} \mbox{ strongly in }
	L^2(\Gamma_2^1\times(0,T)).
	$$
	
	{\it Step 6}. 
	Combining the results of Steps 3-5 and noting that
	$$
	\partial\Omega-\Gamma_2^0
	=
	\Gamma_1^{-1}\cup\Gamma_1^1\cup\Gamma_2^1 \mbox{ a.e.~on }\partial\Omega,
	$$
	we obtain
	$$
	\frac{\partial E_\delta^0\vp_\delta}{\partial\nu_A}
	\to
	\frac{\partial\vp}{\partial\nu_A} \mbox{ strongly in }
	L^2((\partial\Omega-\Gamma_2^0)\times(0,T)).
	$$
	Therefore, \eqref{09.17.8} holds, and the proof is complete.
\end{proof}

\begin{lemma}\label{09.19.L3}
	Let $\vp^0\in C_0^\iy(\Om)$ and $\vp_\de\ (0<\de<\de_0)$ be the solution defined in \eqref{09.16.28}, where $\de_0$ is defined in \eqref{09.17.6}.  Then 
	\begin{equation*}
		\left\|\f{\pt\vp_\de}{\pt \nu_{A_\de}}\right\|_{L^2((\pt\Om_\de-\ol{\Ga_{\de,2}^\de})\ts (0,T))}\leq C\|\vp^0\|_{H_0^1(\Om;w)}, 
	\end{equation*}
	where the positive constant $C$ depends only on $\al$ and $T$.  
\end{lemma}

\begin{proof}
	We divide the proof into three steps.
	
	{\it Step 1}. We estimate on $\Gamma_{\delta,1}^{-1}$.
	
	Let $\zeta=\zeta(x_1)\in C^\infty(\mathbb R)$ satisfy
	$$
	0\leq \zeta\leq 1,\quad 
	\zeta=1\mbox{ on }(-\infty,0),\quad
	\zeta=0\mbox{ on }(1,+\infty),\quad 
	|\zeta'|\leq C \mbox{ and } |\zeta''|\leq C\mbox{ on }\R, 
	$$
	where the positive constants $C$ are absolute. 
	
	Multiplying \eqref{09.16.27} by
	$$
	-\zeta\overline{\partial_{x_1}\vp_\de} -\zeta' \ol\vp_\de, 
	$$
	we integrate over $Q_\de$, then, 
	\begin{equation*}
		\begin{split} 
		0
		&=\iint_{Q_\de} \im \pt_t\vp_\de \left(-\zeta \ol{\pt_{x_1}\vp_\de}-\zeta'\ol\vp_\de\right)\df x\df t-\iint_{Q_\de}\Div(A_\de\nabla\vp_\de)\left( -\zeta \ol{\pt_{x_1}\vp_\de}-\zeta'\ol\vp_\de\right)\df x\df t\\
		&=:A_1+A_2. 
		\end{split} 
	\end{equation*}
	
	By integration by parts in time, we have
	\begin{equation*}
		\begin{aligned}
			A_1
			&=-\im \iint_{Q_\de} \zeta \ol{\pt_t\vp_\de}\pt_{x_1}\vp_\de \df x\df t-\im \int_{\Om_\de} \zeta \vp_\de \ol{\pt_{x_1}\vp_\de}\df x\bigg|_{t=0}^{t=T}\\
			&\hspace{4.5mm}-\im \iint_{Q_\de}\zeta' \ol{\pt_t \vp_\de}  \vp_\de\df x\df t-\im \iint_{Q_\de} \zeta' \ol\vp_\de \pt_t\vp_\de\df x\df t. 
		\end{aligned}
	\end{equation*}
	Then 
	\begin{equation*}
		\begin{split}
			&-\im\iint_{Q_\de} \zeta \pt_t\vp_\de \ol{\pt_{x_1}\vp_\de} \df x\df t-\im \iint_{Q_\de} \zeta'\ol\vp_\de \pt_t\vp_\de \df x\df t\\
			&+\im\iint_{Q_\de} \zeta \ol{\pt_t\vp_\de} \pt_{x_1}\vp_\de\df x\df t+\im\iint_{Q_\de} \zeta'\ol{\pt_t\vp_\de} \vp_\de \df x\df t\\
			&=-\im \int_{\Om_\de} \zeta \vp_\de \ol{\pt_{x_1}\vp_\de}\df x\bigg|_{t=0}^{t=T}-\im \iint_{Q_\de} \zeta'\ol\vp_\de \pt_t\vp_\de\df x\df t, 
		\end{split}
	\end{equation*}
	and hence
	\begin{equation*}
		\begin{split}
			2\Real A_1
			&=-\im \int_{\Om_\de} \zeta\vp_\de \ol{\pt_{x_1}\vp_\de}\df x\bigg|_{t=0}^{t=T}-\im \iint_{Q_\de}\zeta'\ol\vp_\de\pt_t\vp_\de\df x\df t
		\end{split}
	\end{equation*}
	Note that 
	\begin{equation*}
		\begin{split}
			-\im \iint_{Q_\de} \zeta' \ol\vp_\de \pt_t\vp_\de\df x\df t
			&=-\iint_{Q_\de} \zeta'\ol\vp_\de \Div(A_\de\nabla\vp_\de)\df x\df t= \iint_{Q_\de}A_\de\nabla\vp_\de \cdot \nabla(\zeta'\ol\vp_\de)\df x\df t, 
		\end{split}
	\end{equation*}
	therefore, 
	\begin{equation}\label{09.20.1}
		2\Real A_1=-\im \int_{\Om_\de} \zeta\vp_\de\ol{\pt_{x_1}\vp_\de}\df x\bigg|_{t=0}^{t=T}+\iint_{Q_\de} A_\de\nabla\vp_\de\cdot \nabla(\zeta'\ol\vp_\de)\df x\df t. 
	\end{equation}
	
	By integration by parts, we have 
	\begin{equation*}
		\begin{split}
			A_2
			&=\iint_{\Sigma_\de }\zeta (A_\de\nabla\vp_\de\cdot \nu) \ol{\pt_{x_1}\vp_\de}\df S\df t-\iint_{Q_\de} (A_\de\nabla\vp\cdot \nabla\zeta) \ol{\pt_{x_1}\vp_\de}\df x\df t\\
			&\hspace{4.5mm}-\iint_{Q_\de} A_\de\nabla\vp_\de \cdot \nabla (\zeta'\ol\vp_\de)\df x\df t-\iint_{Q_\de}\zeta A_\de\nabla\vp_\de \cdot \nabla(\ol{\pt_{x_1}\vp_\de})\df x\df t, 
		\end{split}
	\end{equation*}
	and 
	\begin{equation*}
		\begin{split}
			&2\Real \iint_{Q_\de}\zeta A_\de\nabla\vp_\de \cdot \nabla(\ol{\pt_{x_1}\vp_\de})\df x\df t\\
			&=\iint_{Q_\de}\zeta A_\de\nabla\vp_\de \cdot \nabla(\ol{\pt_{x_1}\vp_\de})\df x\df t+\iint_{Q_\de}\zeta A_\de\nabla \ol\vp_\de \cdot \nabla({\pt_{x_1}\vp_\de})\df x\df t\\
			&=-\iint_{\Ga_{\de,1}^{-1}\ts (0,T)} \zeta \nabla \vp_\de\cdot A_\de\nabla\ol \vp_\de\df S\df t-\iint_{Q_\de} \zeta' \nabla\vp_\de \cdot A_\de\nabla\ol \vp_\de\df x\df t, 
		\end{split}
	\end{equation*}
	then,  from (see Lemma \ref{09.17.L1})
	\begin{equation}\label{09.20.3}
		\begin{split}
			\pt_{x_1}\vp_\de
			&=\sum_{n,k=1}^\iy e^{\im \la_{nk}^\de t} \vp_{\de,nk}^0 \pt_{x_1}\Phi_n^{(1)}\Phi_{\de,k}^{(2)}=0 \mbox{ on } \Ga_{\de,2}^\de\cup \Ga_{\de,2}^1, \\
			\pt_{x_2}\vp_\de 
			&=\sum_{n,k=1}^\iy e^{\im \la_{nk}^\de t}\vp_{\de,nk}^0 \Phi_n^{(1)}\pt_{x_2}\Phi_{\de,k}^{(2)}=0 \mbox{ on }\Ga_{\de,1}^{-1}\cup \Ga_{\de,1}^{1}, 
		\end{split}
	\end{equation}   
	we obtain  $\nabla\vp_\de\cdot A_\de\nabla\ol\vp_\de=|\pt_{x_1}\vp_\de|^2$ on $\Ga_{\de,1}^{-1}\ts (0,T)$ and $\nabla\vp_\de \cdot A_\de\nabla\ol \vp_\de=0$ on $\Sigma_\de-[\Ga_{\de,1}^{-1}\ts (0,T)]$ and  
	\begin{equation}\label{09.20.2}
		\begin{split}
			2\Real A_2
			&=- \iint_{\Ga_{\de,1}^{-1}\ts (0,T)}|\pt_{x_1}\vp_\de |^2\df S\df t-2\Real \iint_{Q_\de} A_\de\nabla\vp_\de \cdot \nabla (\zeta'\ol\vp_\de)\df x\df t\\
			&\hspace{4.5mm}-2\Real \iint_{Q_\de} (A_\de\nabla\vp_\de\cdot \nabla\zeta) \ol{\pt_{x_1}\vp_\de}\df x\df t+\iint_{Q_\de} \zeta' \nabla\vp_\de \cdot A_\de\nabla\ol \vp_\de\df x\df t.  
		\end{split}
	\end{equation}
	
	Combining \eqref{09.20.1} and \eqref{09.20.2} and Lemma \ref{09.20.L1}  and 
	\begin{equation*}
		\f{\pt\vp_\de}{\pt\nu_{A_\de}}=\nabla\vp_\de \cdot A_\de\nu=-\pt_{x_1}\vp_\de, 
	\end{equation*}
	we obtain 
	\begin{equation*}
		\begin{split}
			\left\|\f{\pt\vp_\de}{\pt\nu_{A_\de}}\right\|_{L^2(\Ga_{\de,1}^{-1}\ts (0,T))}^2
			&\leq C\left(\|\vp_\de(T)\|_{H_0^1(\Om_\de;w_\de)}^2+\|\vp_\de(0)\|_{H_0^1(\Om_\de;w_\de)}^2+\|\vp_\de\|_{L^2(0,T; H_0^1(\Om_\de;w_\de))}^2\right)\\
			&\leq C(1+T)\|\vp^0\|_{H_0^1(\Om_\de;w_\de)}^2=C(1+T)\|\vp^0\|_{H_0^1(\Om;w)}^2, 
		\end{split}
	\end{equation*}
	where the positive constants $C$ depends only on $\al$.  
	
	\textit{Step 2}. We estimate the boundary term on $\Gamma_{\delta,1}^{1}$.
	The estimate follows from the same argument as in Step 1, with the change of
	variable $x_1\mapsto -x_1$.
	
	{\it Step 3}. We estimate on $\Ga_{\de,2}^1$. 
	
	Let $\zeta=\zeta(x_2)\in C^\iy(\R)$ satisfying 
	\begin{equation*}
		\begin{split}
			0\leq \zeta \leq 1, \quad \zeta=1 \mbox{ on }\left(\f{1}{2}, +\iy\right),\quad \zeta=0 \mbox{ on } \left(-\iy, \f{1}{4}\right), \quad |\zeta'|\leq C\mbox{ and }|\zeta''|\leq C\mbox{ on }\R, 
		\end{split}
	\end{equation*}
	where the positive constants $C$ are absolute.

	Multiplying \eqref{09.16.27} by 
	\begin{equation*}
		-\zeta\ol{\pt_{x_2}\vp_\de}-\zeta'\ol\vp_\de, 
	\end{equation*}
	we integrate over $Q_\de$, then 
	\begin{equation*}
		\begin{split}
			0
			&=\iint_{Q_\de} \im \pt_t\vp_\de (-\zeta \ol{\pt_{x_2}\vp_\de}-\zeta'\ol\vp_\de)\df x\df t-\iint_{Q_\de} \Div(A_\de\nabla\vp_\de)(-\zeta \ol{\pt_{x_2}\vp_\de}-\zeta'\ol\vp_\de)\df x\df t\\
			&=: B_1+B_2. 
		\end{split}
	\end{equation*}
	
	By integration by parts in time, we have
	\begin{equation*}
		\begin{split}
			B_1
			&=-\im \iint_{Q_\de} \zeta \ol{\pt_t\vp_\de} \pt_{x_2}\vp_\de \df x\df t-\im \int_{\Om_\de} \zeta\vp_\de \ol{\pt_{x_2}\vp_\de}\df x\bigg|_{t=0}^{t=T}\\
			&\hspace{4.5mm} -\im\iint_{Q_\de} \zeta' \vp_\de \ol{\pt_t\vp_\de} \df x\df t-\im \iint_{Q_\de} \zeta' \ol\vp_\de \pt_t\vp_\de \df x\df t, 
		\end{split}
	\end{equation*}
	then 
	\begin{equation}\label{09.20.4}
		\begin{split}
			2\Real B_1
			&=-\im \iint_{Q_\de}\zeta  \pt_t\vp_\de  \ol{\pt_{x_2}\vp_\de}\df x\df t-\im \iint_{Q_\de} \zeta'\ol\vp_\de\pt_t\vp_\de\df x\df t\\
			&\hspace{4.5mm}+\im\iint_{Q_\de} \zeta \ol{\pt_t\vp_\de} \pt_{x_2}\vp_\de\df x\df t+\im \iint_{Q_\de} \zeta' \vp_\de \ol{\pt_t\vp_\de}\df x\df t\\
			&=-\im \int_{\Om_\de} \zeta \vp_\de \ol{\pt_{x_2}\vp_\de}\df x\bigg|_{t=0}^{t=T}-\im\iint_{Q_\de} \zeta' \ol\vp_\de \pt_t\vp_\de \df x\df t\\
			&=-\im \int_{\Om_\de} \zeta \vp_\de \ol{\pt_{x_2}\vp_\de}\df x\bigg|_{t=0}^{t=T}+\iint_{Q_\de} A_\de\nabla\vp_\de \cdot \nabla (\zeta'\ol\vp_\de)\df x\df t. 
		\end{split}
	\end{equation}
	
	By integration by parts, we have 
	\begin{equation*}
		\begin{split}
			B_2
			&=\iint_{\Sigma_\de } \zeta(A_\de\nabla\vp_\de\cdot \nu)\ol{\pt_{x_2}\vp_\de}\df S\df t-\iint_{Q_\de} (A_\de\nabla\vp_\de\cdot \nabla\zeta)\ol{\pt_{x_2}\vp_\de}\df x\df t\\
			&\hspace{4.5mm}-\iint_{Q_\de}A_\de\nabla\vp_\de\cdot \nabla(\zeta'\ol\vp_\de)\df x\df t-\iint_{Q_\de}\zeta (A_\de\nabla\vp_\de\cdot \nabla \ol{\pt_{x_2}\vp_\de})\df x\df t, 
		\end{split}
	\end{equation*}
	and 
	\begin{equation*}
		\begin{split}
			&2\Real \iint_{Q_\de}\zeta (A_\de\nabla\vp_\de\cdot \nabla \ol{\pt_{x_2}\vp_\de})\df x\df t\\
			&=\iint_{Q_\de} \zeta (A_\de\nabla\vp_\de\cdot \nabla \ol{\pt_{x_2}\vp_\de})\df x\df t+\iint_{Q_\de} \zeta (A_\de \nabla\ol\vp_\de\cdot \nabla\pt_{x_2}\vp_\de)\df x\df t\\
			&=\iint_{\Ga_{\de,2}^1\ts (0,T)}\zeta A_\de\nabla\vp_\de\cdot \nabla\ol \vp_\de\df S\df t-\iint_{Q_\de} \zeta' x_2^\al |\pt_{x_2}\vp_\de|^2\df x\df t-\al \iint_{Q_\de}\zeta   x_2^{\al-1} |\pt_{x_2}\vp_\de|^2\df x\df t-\iint_{Q_\de}\zeta'
\left|\partial_{x_1}\vp_\de\right|^2\,dxdt., 
		\end{split}
	\end{equation*}
	therefore, from \eqref{09.20.3}, we obtain 
	\begin{equation}\label{09.20.5}
		\begin{split}
			2\Real B_2
			&=\iint_{\Ga_{\de,2}^1\ts (0,T)}|\pt_{x_2}\vp_\de|^2\df S\df t-2\Real \iint_{Q_\de}A_\de \nabla\vp_\de\cdot \nabla(\zeta' \ol\vp_\de)\df x\df t\\
			&\hspace{4.5mm}-2\Real \iint_{Q_\de} (A_\de\nabla\vp_\de\cdot \nabla\zeta)\ol{\pt_{x_2}\vp_\de}\df x\df t+\iint_{Q_\de} \zeta' x_2^\al |\pt_{x_2}\vp_\de|^2\df x\df t\\
			&\hspace{4.5mm}+\al \iint_{Q_\de} \zeta x_2^{\al-1}|\pt_{x_2}\vp_\de|^2\df x\df t-\iint_{Q_\de}\zeta'
\left|\partial_{x_1}\vp_\de\right|^2\,dxdt.. 
		\end{split}
	\end{equation}
	
	Combining \eqref{09.20.4} and \eqref{09.20.5} and $A_\de\nabla\vp_\de\cdot\nu=\pt_{x_2}\vp_\de$ on $\Ga_{\de,2}^1$, and Lemma \ref{09.20.L1}, we get
	\begin{equation*}
		\begin{split}
			\left\|\f{\pt\vp_\de}{\pt\nu_{A_\de}}\right\|_{L^2(\Ga_{\de,2}^1\ts (0,T))}^2
			&\leq C\left(\|\vp_\de(0)\|_{H_0^1(\Om_\de;w_\de)}^2+\|\vp_\de(T)\|_{H_0^1(\Om_\de;w_\de)}^2+\|\vp_\de\|_{L^2(0,T; H_0^1(\Om_\de;w_\de))}^2\right), 
		\end{split}
	\end{equation*}
	where the positive constant $C$ depends only on $\al$ and $T$. 
	
	{\it Step 4}. From Steps 1-3, we get the proof. 
\end{proof}

\begin{corollary}\label{09.20.C1}
	Let $\vp^0\in C_0^\iy(\Om)$ and $\vp$ be the solution defined in \eqref{09.16.11}. Then 
	\begin{equation*}
		\left\|\f{\pt\vp}{\pt \nu_{A}}\right\|_{L^2((\pt\Om-\ol{\Ga_{2}^0})\ts (0,T))}\leq C\|\vp^0\|_{H_0^1(\Om;w)}, 
	\end{equation*}
	where the positive constant $C$ depends only on $\al$ and $T$.  
\end{corollary}

\begin{proof}
	By Theorem \ref{09.17.T2}, we have
	$$
	\frac{\partial E_\delta^0\vp_\delta}{\partial\nu_A}
	\to
	\frac{\partial\vp}{\partial\nu_A} \mbox{ strongly in }
	L^2((\partial\Omega-\Gamma_2^0)\times(0,T)).
	$$
	Moreover, Lemma \ref{09.19.L3} gives the uniform estimate
	$$
	\left\|
	\frac{\partial E_\delta^0\vp_\delta}{\partial\nu_A}
	\right\|_
	{L^2((\partial\Omega-\Gamma_2^0)\times(0,T))}
	\leq
	C\|\vp^0\|_{H_0^1(\Omega;w)},
	$$
	where $C$ is independent of $\delta$. 
	Passing to the limit as $\delta\to0^+$, we obtain
	$$
	\left\|
	\frac{\partial\vp}{\partial\nu_A}
	\right\|_
	{L^2((\partial\Omega-\overline{\Gamma_2^0})\times(0,T))}
	\leq
	C\|\vp^0\|_{H_0^1(\Omega;w)} .
	$$
	This proves the corollary.
\end{proof}

\begin{remark}\label{09.21.R1}
	The estimate in Corollary \ref{09.20.C1} cannot be obtained directly by the classical hidden regularity argument. Indeed, due to the boundary degeneracy of the operator, the domain of the degenerate elliptic operator does not generally provide the $H^2(\Omega)$-regularity required for the classical Rellich multiplier method. Therefore, the integration by parts procedure leading to the boundary trace estimate cannot be directly justified for the weak solution of the degenerate problem. Instead, we derive the estimate by considering a family of uniformly elliptic approximations and passing to the limit through the convergence result established in Theorem \ref{09.17.T2}. \end{remark}

\section{Observability}\label{S4}

\subsection{Observability for uniformly approximations}

In this subsection, we consider the case $\delta\in(0,\de_0)$,
$\vp^0\in C_0^\infty(\Omega)$, and $f_\delta=0$, where $\delta_0>0$
is defined in \eqref{09.17.6}. The weak solution of
\eqref{09.16.27} corresponding to the initial data $(\vp^0,0)$ is
defined by Remark \ref{09.16.R2} or equivalently by the spectral
representation \eqref{09.16.29}. 

By Lemma \ref{09.17.L1}, the weak solution $\vp_\delta$ possesses the
required regularity. Therefore, the integration by parts arguments
below are justified as in the classical uniformly elliptic case.

Let
$$
H=x
$$
be the vector field. Define
$$
\Gamma_\de^+=\{x\in\partial\Omega_\delta:H\cdot\nu>0\},
\quad
\Sigma_\de^+=\Gamma_\de^+\times(0,T).
$$
Then
\begin{equation}\label{09.21.2}
\Gamma_\de^+=\partial\Omega_\delta-\overline{\Gamma_{\delta,2}^{\delta}}, \mbox{ and } 
H\cdot\nu=1 \mbox{ on }\Gamma_\de^+.
\end{equation}

Multiplying \eqref{09.16.27} by
$$
H\cdot\nabla\overline{\vp_\delta}+\overline{\vp_\delta},
$$
and integrating over $Q_\delta$, we obtain
\begin{equation*}
	\begin{split}
		0
		&=\im\iint_{Q_\de} \pt_t\vp_\de\left(H\cdot \nabla\ol\vp_\de+\ol\vp_\de\right)\df x\df t-\iint_{Q_\de}[\Div(A_\de\nabla\vp_\de)]\left( H\cdot \nabla\ol\vp_\de+\ol\vp_\de\right)\df x\df t\\
		&=:S_1+S_2. 
	\end{split}
\end{equation*}

From integration by parts, we obtain 
\begin{equation*}
	\begin{split}
		&2\im \iint_{Q_\de}\pt_t\vp_\de (H\cdot \nabla\ol\vp_\de)\df x\df t-2\im \iint_{Q_\de} \pt_t\ol\vp_\de (H\cdot \nabla\vp_\de)\df x\df t\\
		&=-2\Image \int_{\Om_\de} \vp_\de (H\cdot \nabla\ol\vp_\de)\df x\bigg|_{t=0}^{t=T}+2\im \iint_{Q_\de}\vp_\de \pt_t\ol\vp_\de\df x\df t-2\im \iint_{Q_\de}\ol\vp_\de \pt_t\vp_\de\df x\df t, 
	\end{split}
\end{equation*}
and then 
\begin{equation}\label{09.21.1}
	\begin{split}
		2\Real S_1
		&=\im \iint_{Q_\de} \pt_t\vp_\de (H\cdot \nabla\ol \vp_\de)\df x\df t-\im \iint_{Q_\de} \pt_t\ol\vp_\de (H\cdot \nabla\vp_\de)\df x\df t\\
		&\hspace{4.5mm}+\im \iint_{Q_\de} (\pt_t\vp_\de)\ol\vp_\de\df x\df t-\im \iint_{Q_\de} (\pt_t\ol\vp_\de) \vp_\de\df x\df t\\
		&=-\Image \int_{\Om_\de} \vp_\de (H\cdot \nabla\ol\vp_\de)\df x\bigg|_{t=0}^{t=T}. 
	\end{split}
\end{equation}

Since 
\begin{equation*}
	\begin{split}
		&\iint_{Q_\de}A_\de\nabla\vp_\de\cdot [(D^2\ol\vp_\de)H]\df x\df t+\iint_{Q_\de} A_\de\nabla\ol\vp_\de\cdot [(D^2\vp_\de)H]\df x\df t\\
		&=\iint_{\Sigma_\de}(A_\de\nabla\vp_\de\cdot \nabla\ol\vp_\de) (H\cdot\nu)\df S\df t-2\iint_{Q_\de} A_\de\nabla\vp_\de\cdot \nabla\ol\vp_\de\df x\df t-\al\iint_{Q_\de}x_2^\al |\pt_{x_2}\vp_\de|^2\df x\df t, 
	\end{split}
\end{equation*}
we obtain 
\begin{equation*}
	\begin{split}
		&-\iint_{Q_\de} [\Div(A_\de\nabla\vp_\de)](H\cdot \nabla\ol\vp_\de)\df x\df t-\iint_{Q_\de} [\Div(A_\de\nabla\ol \vp_\de)](H\cdot \nabla\vp_\de)\df x\df t\\
		&=-\iint_{\Sigma_\de} (A_\de\nabla\vp_\de\cdot\nu)(H\cdot\nabla\ol\vp_\de)\df S\df t-\iint_{\Sigma_\de} (A_\de\nabla\ol\vp_\de\cdot\nu)(H\cdot\nabla\vp_\de)\df S\df t\\
		&\hspace{4.5mm}+\iint_{\Sigma_\de}(A_\de\nabla\vp_\de\cdot \nabla\ol\vp_\de) (H\cdot\nu)\df S\df t-\al \iint_{Q_\de} x_2^\al |\pt_{x_2}\vp_\de|^2\df x\df t\\
		&=-\iint_{\Sigma_\de}\left|\f{\pt\vp_\de}{\pt\nu}\right|^2(A_\de\nu\cdot\nu)(H\cdot\nu)\df S\df t-\al\iint_{Q_\de}x_2^\al |\pt_{x_2}\vp_\de|^2\df x\df t
	\end{split}
\end{equation*}
by $\nabla\vp_\de=(\nabla\vp_\de\cdot\nu)\nu$ on $\Sigma$. Hence, from 
\begin{equation*}
	\begin{split}
		&-\iint_{Q_\de}[\Div(A_\de\nabla\vp_\de)]\ol\vp_\de \df x\df t-\iint_{Q_\de} [\Div(A_\de\nabla\ol\vp_\de)]\vp_\de\df x\df t =2\iint_{Q_\de} A_\de \nabla\vp_\de\cdot \nabla\ol\vp_\de\df x\df t, 
	\end{split}
\end{equation*}
we get 
\begin{equation*}
	\begin{split}
		2\Real S_2
		&=-\iint_{\Sigma_\de}\left|\f{\pt\vp_\de}{\pt\nu}\right|^2(A_\de\nu\cdot\nu)(H\cdot\nu)\df S\df t\\
		&\hspace{4.5mm}+2\iint_{Q_\de} A_\de \nabla\vp_\de\cdot \nabla\ol\vp_\de\df x\df t-\al\iint_{Q_\de}x_2^\al |\pt_{x_2}\vp_\de|^2\df x\df t. 
	\end{split}
\end{equation*}
Combining this and \eqref{09.21.1}, and Lemma \ref{09.20.L1}, and $H\cdot \nu=-\de$ on $\Ga_{\de,2}^\de$ and \eqref{09.21.2}, and $A_\de\nu\cdot\nu=1$ on $\Ga_{\de}^+$ and $A_\de\nu\cdot\nu=\de^\al$ on $\Ga_{\de,2}^\de$, and \eqref{09.15.3}, we obtain 
\begin{equation*}
	\begin{split}
		(2-\al)T\|\vp_\de(0)\|_{H_0^1(\Om_\de;w_\de)}^2
		&\leq \Image \int_{\Om_\de} \vp_\de (H\cdot \nabla\ol\vp_\de)\df x\bigg|_{t=0}^{t=T}+\iint_{\Sigma_\de}\left|\f{\pt\vp_\de}{\pt\nu}\right|^2(A_\de\nu\cdot\nu)(H\cdot\nu)\df S\df t\\
		&\leq \iint_{\Sigma_\de^+}\left|\f{\pt\vp_\de}{\pt\nu}\right|^2\df S\df t+\|\vp_\de(0)\|_{L^2(\Om_\de)}^2+2\|\vp_\de(0)\|_{H_0^1(\Om_\de;w_\de)}^2\\
		&\leq \iint_{\Sigma_\de^+}\left|\f{\pt\vp_\de}{\pt\nu}\right|^2\df S\df t+C\|\vp_\de(0)\|_{H_0^1(\Om_\de;w_\de)}^2, 
	\end{split}
\end{equation*}
therefore, 
\begin{equation*}
	\begin{split}
		\left[(2-\al)T-C\right]\|\vp_\de^0\|_{H_0^1(\Om_\de;w_\de)}^2\leq \iint_{\Sigma_\de^+}\left|\f{\pt\vp_\de}{\pt\nu}\right|^2\df S\df t, 
	\end{split}
\end{equation*}
where the positive constant $C$ depends only on $\al$. Therefore, if 
\begin{equation}\label{09.21.5}
T>T_0:=\frac{C}{2-\alpha},
\end{equation} 
we obtain the following observability estimate.

\begin{theorem}\label{09.21.T1}
	Let $\delta\in(0,\delta_0)$ and $\vp^0\in C_0^\infty(\Omega)$.
	Let $\vp_\delta$ be the weak solution of \eqref{09.16.27}
	with respect to  $(\vp^0,0)$, whose representation is given by
	\eqref{09.16.29}. Then there exists a positive constant $C$,
	depending only on $\alpha$, such that for every
	$$
	T>T_0,
	$$
	one has
	\begin{equation}\label{09.21.3}
		\left[(2-\alpha)T-C\right]
		\|\vp_\delta^0\|_{H_0^1(\Omega_\delta;w_\delta)}^2
		\leq
		\iint_{\Sigma_\delta^+}
		\left|\frac{\partial\vp_\delta}{\partial\nu}\right|^2
		\,dSdt .
	\end{equation} 
\end{theorem}

\subsection{Observability for the original equation}

Define
\begin{equation*}
	\Ga^+=\{x\in\pt\Om\colon x\cdot\nu>0\}, \quad \Sigma^+=\Ga^+\ts (0,T). 
\end{equation*}
Then $\Ga^+=\pt\Om-\Ga_2^0,\mbox{ and } x\cdot \nu=1 \mbox{ on }\Ga^+$. 

From Theorems \ref{09.21.T1} and \ref{09.17.T2}, we obtain the
following observability estimate for the degenerate Schrödinger
equation.

\begin{theorem}\label{09.21.T2}
	Let $\vp^0\in C_0^\infty(\Omega)$ and let $\vp$ be the weak solution
	of \eqref{09.15.1} with respect to $(\vp^0,0)$, whose representation
	is given by \eqref{09.16.11}. Then there exists a positive constant
	$C$, depending only on $\alpha$, such that for every
	$$
	T>T_0,
	$$
	one has
	\begin{equation}\label{09.21.4}
		\left[(2-\alpha)T-C\right]
		\|\vp^0\|_{H_0^1(\Omega;w)}^2
		\leq
		\iint_{\Sigma^+}
		\left|
		\frac{\partial\vp}{\partial\nu_A}
		\right|^2
		\,dSdt. 
	\end{equation} 
\end{theorem}

From Theorem \ref{09.21.T2}, we obtain the following theorem. 

\begin{theorem}\label{09.21.T3}
	Let $\vp^0\in H_0^1(\Om;w)$ and let $\vp$ be the weak solution
	of \eqref{09.15.1} with respect to $(\vp^0,0)$, whose representation
	is given by \eqref{09.16.11}. Then there exists a positive constant
	$C$, depending only on $\alpha$, such that for every
	$$
	T>T_0,
	$$
	one has
	\begin{equation}\label{09.21.6}
		\left[(2-\alpha)T-C\right]
		\|\vp^0\|_{H_0^1(\Omega;w)}^2
		\leq
		\iint_{\Sigma^+}
		\left|
		\frac{\partial\vp}{\partial\nu_A}
		\right|^2
		\,dSdt, 
	\end{equation} 
	where the positive  constant $T_0$ is defined in \eqref{09.21.5}. 
\end{theorem}

\begin{proof}
	We prove this theorem by the following steps. 
	
	{\it Step 1}. 
	Let $\vp^0\in H_0^1(\Om;w)$. Then there exists a sequence $\{\vp_n^0\}_{n\in\N^*}\s C_0^\iy(\Om)$ such that 
	\begin{equation}\label{09.21.11}
		\vp_n^0\ra \vp^0 \mbox{ strongly in }H_0^1(\Om;w). 
	\end{equation}
	
	Denote $\vp_n$ be the weak solution of \eqref{09.15.1} with respect to $(\vp_n^0,0)$, then, from Theorem \ref{09.21.T2}, there exists a positive constant $C$, depending only on $\al$, such that for every $T>T_0$ we have 
	\begin{equation}\label{09.21.10}
		[(2-\al)T-C]\|\vp_n^0\|_{H_0^1(\Om;w)}^2\leq \iint_{\Sigma^+}\left|\f{\pt\vp_n}{\pt \nu_A}\right|^2\df S\df t. 
	\end{equation}
	
	{\it Step 2}. From \eqref{09.16.4}, we obtain 
	\begin{equation*}
		\esssup_{t\in[0,T]}\|\vp_n(t)-\vp(t)\|_{H_0^1(\Om;w)}+\|\pt_t(\vp_n-\vp)\|_{L^2(0,T; H^{-1}(\Om;w))}\leq \|\vp_n^0-\vp^0\|_{H_0^1(\Om;w)}. 
	\end{equation*}
	Hence,   
	\begin{equation}\label{09.21.7}
		\begin{split}
			\vp_n\ra \vp \mbox{ strongly in } L^\iy(0,T; H_0^1(\Om;w)). 
		\end{split}
	\end{equation}
	
	Note that from Corollary \ref{09.20.C1} we obtain 
	\begin{equation*}
		\left\|\f{\pt (\vp_n-\vp_m)}{\pt\nu_A}\right\|_{L^2(\Ga^+\ts (0,T))}\leq C\|\vp_n^0-\vp_m^0\|_{H_0^1(\Om;w)}, 
	\end{equation*}
	where the positive constant $C$ depends only on $\al$ and $T$. Then there exists $z\in L^2(\Ga^+\ts (0,T))$ such that 
	\begin{equation}\label{09.21.9}
		\f{\pt \vp_n}{\pt \nu_A}\ra z\mbox{ strongly in } L^2(\Ga^+\ts (0,T)). 
	\end{equation}
	
		Now, we identify the limit $z$ with the conormal derivative of
	$\vp$. i.e., $z=\f{\pt\vp}{\pt \nu_A}$. 
	
	Let  $
	\psi\in C_0^\infty(\Ga_1^{-1}\times(0,T))$. 
	Then there exists $\Psi\in C^\infty(\overline{Q})$ such that
	\begin{equation*} 
			\Psi=\psi 
			\mbox{ on }\Ga_1^{-1}\times(0,T),\quad 
			\dist\left(\supp \Psi, 
			[(\partial\Om-\Ga_1^{-1})\times(0,T)]
			\cup [\Om\times\{0,T\}]\right)>2\de_0'>0. 
	\end{equation*}
	 
	Since $\vp_n|_{\Om_{\de_0'}\ts (0,T)}\in L^2(0,T; H^2(\Om_{\de_0'}))$  (using Step 1 in the proof of Lemma \ref{08.05.L2}, similar to the proof of Lemma \ref{09.17.L1}), multiplying the equation
	\eqref{09.15.1} by $\overline{\Psi}$ and integrating by parts,
	we obtain
	\begin{equation*}
		\begin{split}
			0
			&=
			\iint_Q
			\im \partial_t\vp_n\overline{\Psi}\,dxdt
			-
			\iint_Q
			\Div(A\nabla\vp_n)\overline{\Psi}\,dxdt\\
			&=
			-\im \iint_Q
			\vp_n\partial_t\overline{\Psi}\,dxdt
			+
			\iint_Q
			A\nabla\vp_n\cdot\nabla\overline{\Psi}\,dxdt 
			-
			\iint_{\partial\Om\times(0,T)}
			(A\nabla\vp_n\cdot\nu)
			\overline{\Psi}\,dSdt .
		\end{split}
	\end{equation*} 
	Therefore, 
	\begin{equation}\label{09.21.8}
		\begin{split}
			&
			\iint_{\Ga_1^{-1}\times(0,T)}
			\frac{\partial\vp_n}{\partial\nu_A}
			\overline{\psi}\,dSdt =
			\iint_Q
			A\nabla\vp_n\cdot\nabla\overline{\Psi}\,dxdt
			-
			\im \iint_Q
			\vp_n\partial_t\overline{\Psi}\,dxdt .
		\end{split}
	\end{equation} 
	Passing to the limit as $n\to\infty$, from
	\eqref{09.21.7}  and the convergence of the conormal derivatives  \eqref{09.21.9},
	we obtain
	\begin{equation}\label{09.22.2}
		\begin{split}
			&
			\iint_{\Ga_1^{-1}\times(0,T)}
			z\overline{\psi}\,dSdt =
			\iint_Q
			A\nabla\vp\cdot\nabla\overline{\Psi}\,dxdt
			-
			\im \iint_Q
			\vp\partial_t\overline{\Psi}\,dxdt .
		\end{split}
	\end{equation} 
	This identity allows us to identify \(z\) as the weak conormal derivative of  
	$\vp$ on $\Ga_1^{-1}\times(0,T)$. Hence,
	\begin{equation}\label{09.22.1}
		z=
		\frac{\partial\vp}{\partial\nu_A} \mbox{ on }\Ga_1^{-1}\times(0,T).
	\end{equation}

	Similarly, by taking $\psi\in C_0^\iy(\Ga_1^1\ts (0,T))$ and 
	$\psi\in C_0^\infty(\Ga_2^1\times(0,T))$,
	we obtain
	\begin{equation*}
		z=
		\frac{\partial\vp}{\partial\nu_A} \mbox{ on }[\Ga_1^1\cup \Ga_2^1]\times(0,T).
	\end{equation*}
	
	Consequently,
	\begin{equation*}
		z=
		\frac{\partial\vp}{\partial\nu_A} \mbox{ on }\Ga^+\times(0,T).
	\end{equation*}

	{\it Step 3}. Combining \eqref{09.21.11} and  \eqref{09.21.10} and \eqref{09.21.9}, we obtain \eqref{09.21.6}. We prove the theorem.  
\end{proof}

\begin{remark}\label{09.22.R1}
	Due to the boundary degeneracy of the diffusion matrix 
	$$
	A=\operatorname{diag}(1,x_2^\alpha),\quad \alpha\in(0,1),
	$$ 
	the weak solution of the degenerate Schrödinger equation may not possess sufficient regularity for the classical conormal derivative to be defined. Therefore, the boundary conormal derivative is introduced in the weak sense through the Green identity.
	
	Let $\psi$ be an admissible test function on
	$\Gamma_1^{-1}\times(0,T)$, and let $\Psi$ be an extension of $\psi$  into $Q$. We define $z$ by
	\begin{equation*}
		\begin{split}
			\iint_{\Gamma_1^{-1}\times(0,T)}
			z\overline{\psi}\df S\df t
			&=
			\iint_Q
			A\nabla\varphi\cdot\nabla\overline{\Psi}\df x\df t -
			\mathrm{i}\iint_Q
			\varphi\partial_t\overline{\Psi}\df x\df t .
		\end{split}
	\end{equation*}
	The above identity uniquely determines $z$ in the corresponding boundary space, and this definition is independent of the choice of the extension $\Psi$. The function $z$ is called the generalized conormal trace of $\varphi$ on $\Gamma_1^{-1}\times(0,T)$.
	
	For sufficiently regular solutions, this generalized trace coincides with the classical conormal derivative 
	$$
	z=(A\nabla\varphi)\cdot\nu
	$$ 
	on $\Gamma_1^{-1}\times(0,T)$. Hence, the above definition extends the classical boundary flux to the weak solution framework. Moreover, Corollary \ref{09.20.C1} gives the hidden boundary regularity of this generalized conormal trace. More precisely, 
	$$
	\frac{\partial\varphi}{\partial\nu_A}
	\in
	L^2\big((\partial\Omega-\overline{\Gamma_2^0})\times(0,T)\big),
	$$ 
	and 
	$$
	\left\|
	\frac{\partial\varphi}{\partial\nu_A}
	\right\|_{L^2((\partial\Omega-\overline{\Gamma_2^0})\times(0,T))}
	\leq
	C\|\varphi^0\|_{H_0^1(\Omega;w)}.
	$$ 
	Therefore, although the classical conormal derivative is not directly available for weak solutions of the degenerate Schrödinger equation, the generalized conormal trace is well-defined and provides the appropriate boundary observation for the observability inequality.
\end{remark}

\section{Appendix}\label{SA}

In this section, we give the proof of Theorems \ref{09.16.T1} and  \ref{09.17.T1}. 

\begin{proof}[Proof of Theorem \ref{09.16.T1}]
	We prove this theorem by the following steps. 
	
	{\it Step 1}. Galerkin method. 
	
	Let $m\in\N^*$. Define 
	\begin{equation}\label{09.16.5}
		\vp^m=\sum_{n=1}^m \vp_n^m(t) \Phi_n, \quad f^m=\sum_{n=1}^m f_n^m(t)\Phi_n \mbox{ with } f_n^m(t)=(f,\Phi_n)_{L^2(\Om)}, 
	\end{equation}
	where $\vp_n^m(t), n=1,\cdots, m$ is the solution of the ordinary differential equation 
	\begin{equation}\label{09.16.6}
		\im \pt_t\vp_n^m(t)+\la_n \vp_n^m(t)=f_n^m(t),   \ n=1,\cdots, m, \quad   \pt_t\vp_n^m\equiv \f{\df\vp_n^m}{\df t},
	\end{equation}
	with initial data
	\begin{equation}\label{09.16.7}
		\vp_n^m(0)=\vp_n^0, \mbox{ with } \vp_n^0=(\vp^0,\Phi_n)_{L^2(\Om)}, \ n=1,\cdots, m. 
	\end{equation}
	Indeed, the equation \eqref{09.16.6} is equivalent to the following equation 
	\begin{equation}\label{09.16.8}
		\im(\pt_t\vp^m, \Phi_n)_{L^2(\Om)}+a(\vp^m, \Phi_n)=(f^m, \Phi_n)_{L^2(\Om)}, \ n=1,\cdots, m, 
	\end{equation}
	where $a(\cdot,\cdot)$ and $\Phi_n$ are defined in \eqref{09.15.5} and \eqref{09.15.6}. 
	
	{\it Step 2}. Energy estimate. 
	
	Multiplying $\ol{\vp_n^m}$ on the both sides of \eqref{09.16.6}, summing $n=1,\cdots, m$, we obtain 
	\begin{equation}\label{09.16.9}
		\begin{split}
			\im \sum_{n=1}^m (\pt_t\vp_n^m)\ol{\vp_n^m}+\sum_{n=1}^m \la_n \vp_n^m \ol{\vp_n^m}=\sum_{n=1}^m f_n^m\ol{\vp_n^m}. 
		\end{split}
	\end{equation}
	Taking conjugate of \eqref{09.16.9}, we obtain 
	\begin{equation*}
		-\im \sum_{n=1}^m \ol{(\pt_t\vp_n^m)}{\vp_n^m}+\sum_{n=1}^m \la_n \vp_n^m \ol{\vp_n^m}=\sum_{n=1}^m \ol{f_n^m}{\vp_n^m}.
	\end{equation*}
	Then 
	\begin{equation*}
		\begin{split} 
			\f{\df}{\df t}\|\vp^m(t)\|_{L^2(\Om)}^2
			&=\sum_{n=1}^m \pt_t\left(\vp_n^m \ol{\vp_n^m}\right)=2\Image \sum_{n=1}^m f_n^m \ol{\vp_n^m}\\
			&=2\Image (f^m, \vp^m)_{L^2(\Om)}\leq \|f^m\|_{L^2(\Om)}^2+\|\vp^m\|_{L^2(\Om)}^2. 
		\end{split} 
	\end{equation*}
	The Gronwall's inequality implies 
	\begin{equation*}
		\begin{split}
			\|\vp^m(t)\|_{L^2(\Om)}^2\leq e^t \left(\|\vp^m(0)\|_{L^2(\Om)}^2+\int_0^t \|f^m(s)\|_{L^2(\Om)}^2\df s\right), \mbox{ for all } t\in [0,T]. 
		\end{split}
	\end{equation*}
	Hence, for all $m\in\N^*$, 
	\begin{equation}\label{09.16.10}
		\esssup_{t\in [0,T]}\|\vp^m(t)\|_{L^2(\Om)}^2\leq e^T\left(\|\vp^0\|_{L^2(\Om)}^2+\|f\|_{L^2(Q)}^2\right). 
	\end{equation}
	
	Multiplying \eqref{09.16.6} by $\ol{\pt_t\vp_n^m}$, summing $n=1,\cdots, m$, we get
	\begin{equation}\label{09.16.12}
		\begin{split}
			\im \sum_{n=1}^m |\pt_t\vp_n^m|^2+\sum_{n=1}^m \la_n \vp_n^m\ol{\pt_t\vp_n^m}=\sum_{n=1}^m f_n^m \ol{\pt_t \vp_n^m}. 
		\end{split}
	\end{equation}
	Taking conjugate of \eqref{09.16.12}, we obtain 
	\begin{equation*}
		\begin{split}
			-\im \sum_{n=1}^m |\pt_t\vp_n^m|^2+\sum_{n=1}^m \la_n \ol{\vp_n^m}\pt_t\vp_n^m=\sum_{n=1}^m \ol{f_n^m}\pt_t\vp_n^m. 
		\end{split}
	\end{equation*}
	Then 
	\begin{equation}\label{09.16.13}
		\begin{split}
			\f{\df}{\df t}a(\vp^m,\vp^m)
			&=2\Real \sum_{n=1}^m \la_n \vp_n^m \ol{\pt_t\vp_n^m}=2\Real \sum_{n=1}^m f_n^m\ol{\pt_t\vp_n^m}\\
			&=2\Real (f^m,\pt_t\vp^m)_{L^2(\Om)}\leq \|f^m\|_{H_0^1(\Om;w)}^2+\|\pt_t\vp^m\|_{H^{-1}(\Om;w)}^2. 
		\end{split}
	\end{equation}
	Note that from \eqref{09.15.6} and  \eqref{09.16.5} and $a(\Phi_n,v)=\la_n(\Phi_n,v)_{L^2(\Om)}$ for all $v\in H_0^1(\Om;w)$, we obtain 
	\begin{equation}\label{09.16.16}
		\begin{aligned}
			&\|\partial_t\varphi^m\|_{H^{-1}(\Omega;w)}\\
			&=
			\sup_{\|v\|_{H_0^1(\Omega;w)}\leq 1}
			\left|
			\langle \partial_t\varphi^m,v\rangle_{H^{-1}(\Omega;w),H_0^1(\Omega;w)}
			\right| =
			\sup_{\|v\|_{H_0^1(\Omega;w)}\leq 1}
			\left|
			(\partial_t\varphi^m,v)_{L^2(\Omega)}
			\right|
			\\
			&=
			\sup_{\|v\|_{H_0^1(\Omega;w)}\leq 1}
			\left|
			\sum_{n=1}^{m}
			\partial_t\varphi_n^m
			(\Phi_n,v)_{L^2(\Omega)}
			\right| =
			\sup_{\|v\|_{H_0^1(\Omega;w)}\leq 1}
			\left|
			\sum_{n=1}^{m}
			\lambda_n^{-1}
			\partial_t\varphi_n^m
			a(\Phi_n,v)
			\right|
			\\
			&\leq
			\sup_{\|v\|_{H_0^1(\Omega;w)}\leq 1}
			\left(
			\sum_{n=1}^{m}
			\lambda_n^{-1}
			|\partial_t\varphi_n^m|^2
			\right)^{\frac12}
			\left(
			\sum_{n=1}^{m}
			\lambda_n^{-1}
			|a(\Phi_n,v)|^2
			\right)^{\frac12} \leq
			\left(
			\sum_{n=1}^{m}
			\lambda_n^{-1}
			|\partial_t\varphi_n^m|^2
			\right)^{\frac12}.
		\end{aligned}
	\end{equation}
	Combining this and \eqref{09.16.13} and \eqref{09.16.6}, we get
	\begin{equation*}
		\begin{split}
			\f{\df}{\df t}a(\vp^m,\vp^m)
			&\leq \|f^m\|_{H_0^1(\Om;w)}^2+\sum_{n=1}^m |\pt_t\vp_n^m|^2\la_n^{-1}\\
			&\leq \|f^m\|_{H_0^1(\Om;w)}^2+2\sum_{n=1}^m \la_n|\vp_n^m|^2+2\sum_{n=1}^m \la_1^{-1} |f_n^m|^2 \leq C\|f^m\|_{H_0^1(\Om;w)}^2+2a(\vp^m,\vp^m), 
		\end{split}
	\end{equation*}
	where the positive constant $C$ depends only on $\al$. 
	By the Gronwall's inequality, we obtain 
	\begin{equation*}
		a(\vp^m(t),\vp^m(t))\leq e^{2t}\left( \|\vp^m(0)\|_{H_0^1(\Om;w)}^2+C\int_0^t \|f^m(s)\|_{H_0^1(\Om;w)}^2\df s\right), \mbox{ for all } t\in [0,T], 
	\end{equation*}
	and then 
	\begin{equation}\label{09.16.15}
		\begin{split}
			\esssup_{t\in [0,T]}\|\vp^m(t)\|_{H_0^1(\Om;w)}^2\leq e^{2T}\left(\|\vp^0\|_{H_0^1(\Om;w)}^2+C\|f\|_{L^2(0,T; H_0^1(\Om;w))}^2\right), 
		\end{split}
	\end{equation}
	where the positive constant $C$ depends only on $\al$. Combining this and \eqref{09.16.6} and  \eqref{09.16.16}, we obtain 
	\begin{equation}\label{09.16.17}
		\begin{split}
			\|\pt_t\vp^m\|_{L^2(0,T; H^{-1}(\Om;w))}^2
			&\leq 2\int_0^Ta(\vp^m,\vp^m)\df t+2\la_1^{-1}\int_0^T \|f^m\|_{L^2(\Om)}^2\df t\\
			&\leq Ce^{2T}(1+T)\left(\|\vp^0\|_{H_0^1(\Om;w)}^2+\|f\|_{L^2(0,T; H_0^1(\Om;w))}^2\right), 
		\end{split}
	\end{equation}
	where the positive constant $C$ depends only on $\al$. 
	
	{\it Step 3}. Approximation or convergence. 
	
	Combining \eqref{09.16.10} and \eqref{09.16.15} and \eqref{09.16.17}, there exists a subsequence of $\{\vp^m\}_{m\in\N^*}$, still denoted by itself, and 
	\begin{equation}\label{09.16.18}
		\wt\vp\in L^\iy(0,T; H_0^1(\Om;w)), \mbox{ with } \pt_t\wt\vp\in L^2(0,T; H^{-1}(\Om;w)), 
	\end{equation} 
	such that 
	\begin{equation}\label{09.16.19}
		\begin{split}
			\vp^m
			&\ra \wt\vp \mbox{ weak-* in } L^\iy(0,T; H_0^1(\Om;w)), \\
			\pt_t\vp^m
			&\ra \pt_t\wt\vp \mbox{ weakly in } L^2(0,T; H^{-1}(\Om;w)). 
		\end{split}
	\end{equation}
	Moreover, 
	\begin{equation}\label{09.16.20}
		\begin{split}
			&\esssup_{t\in [0,T]}\|\wt\vp(t)\|_{H_0^1(\Om;w)}^2+\|\pt_t\wt\vp\|_{L^2(0,T; H^{-1}(\Om;w))}^2\\
			&\leq Ce^{2T}(1+T)\left(\|\vp^0\|_{H_0^1(\Om;w)}^2+\|f\|_{L^2(0,T; H_0^1(\Om;w))}^2\right), 
		\end{split}
	\end{equation}
	where the positive constant $C$ depends only on $\al$.  
	Moreover, 
	\begin{equation}\label{09.16.24}
		\wt\vp\in C([0,T]; L^2(\Om)). 
	\end{equation}
	
	{\it Step 4}. We prove $\wt\vp$ is a weak solution of \eqref{09.16.1}. 
	
	Let $k\in\N^*$. Denote 
	\begin{equation}\label{09.16.21}
		\psi^k(x,t)=\sum_{n=1}^k \psi_n^k(t) \Phi_n(x), \mbox{ where $\psi_n^k(t)$ are smooth functions for }  n=1,\cdots, k. 
	\end{equation}
	Multiply $\psi_n^k(t)$ on the both sides of \eqref{09.16.8}, integrating on $[0,T]$, for all $m\geq k$, we obtain 
	\begin{equation}\label{09.16.23}
		\begin{split}
			\im \int_0^T (\pt_t\vp^m, \psi^k)_{H^{-1}(\Om;w), H_0^1(\Om;w)}\df t+\int_0^T a(\vp^m, \psi^k)\df t=\int_0^T (f^m, \psi^k)_{L^2(\Om)}\df t. 
		\end{split}
	\end{equation}
	Together with this and \eqref{09.16.19}, we get
	\begin{equation*}
		\begin{split}
			\im \int_0^T \lg\pt_t\wt\vp, \psi^k\rg_{H^{-1}(\Om;w), H_0^1(\Om;w)}\df t+\int_0^T a(\wt\vp, \psi^k)\df t=\int_0^T (f, \psi^k)_{L^2(\Om)}\df t. 
		\end{split}
	\end{equation*}
	Since $
	C^\infty([0,T])
	\otimes
	\mathrm{span}\{\Phi_n\colon n\in\mathbb N^\ast\}$ 
	is dense in
	$L^2(0,T;H_0^1(\Om;w))$,
	we obtain
	\begin{equation}\label{09.16.22}
		\im \int_0^T \lg\pt_t\wt\vp, \psi\rg_{H^{-1}(\Om;w), H_0^1(\Om;w)}\df t+\int_0^T a(\wt\vp, \psi)\df t=\int_0^T (f, \psi)_{L^2(\Om)}\df t 
	\end{equation}
	for all $\psi\in L^2(0,T; H_0^1(\Om;w))$. 
	
	For each $\zeta\in C_0^\iy(0,T)$ and $v\in H_0^1(\Om;w)$, from \eqref{09.16.22}, we obtain 
	\begin{equation*} 
		\im \int_0^T\zeta  \lg\pt_t\wt\vp, v\rg_{H^{-1}(\Om;w), H_0^1(\Om;w)}\df t+\int_0^T\zeta  a(\wt\vp, v)\df t=\int_0^T \zeta (f, v)_{L^2(\Om)}\df t. 
	\end{equation*}
	This shows 
	\begin{equation*}
		\lg \im \pt_t\wt\vp, v\rg_{H^{-1}(\Om;w), H_0^1(\Om;w)}+a(\wt\vp, v)=(f, v)_{L^2(\Om)}
	\end{equation*}
	for a.e.~$t\in [0,T]$ and all $v\in H_0^1(\Om;w)$. 
	
	Now, we show $\wt\vp(0)=\vp^0$. 
	
	From \eqref{09.16.23}, we obtain 
	\begin{equation*}
		-\im\int_0^T(\vp^m, \pt_t\psi^k)_{L^2(\Om)}\df t+\int_0^Ta(\vp^m, \psi^k)\df t=\int_0^T(f^m,\psi^k)_{L^2(\Om)}\df t+(\im\vp^m(0),\psi^k(0))_{L^2(\Om)}
	\end{equation*}
	for all $\psi^k$ in \eqref{09.16.21} satisfying $\psi^k(T)=0$. Together with this and \eqref{09.16.7} and \eqref{09.16.19}, we get
	\begin{equation*}
		\begin{split}
			-\im\int_0^T(\wt \vp, \pt_t\psi^k)_{L^2(\Om)}\df t+\int_0^Ta(\wt \vp, \psi^k)\df t=\int_0^T(f,\psi^k)_{L^2(\Om)}\df t+(\im\vp^0,\psi^k(0))_{L^2(\Om)}
		\end{split}
	\end{equation*}
	for all $\psi^k$ in \eqref{09.16.21} satisfying $\psi^k(T)=0$. Moreover, 
	\begin{equation}\label{09.16.25}
		\begin{split}
			-\im\int_0^T(\wt \vp, \pt_t\psi)_{L^2(\Om)}\df t+\int_0^Ta(\wt \vp, \psi)\df t=\int_0^T(f,\psi)_{L^2(\Om)}\df t+(\im\vp^0,\psi(0))_{L^2(\Om)}
		\end{split}
	\end{equation}
	for all $\psi\in C^1([0,T]; H_0^1(\Om;w))$ in \eqref{09.16.21} satisfying $\psi(T)=0$. From \eqref{09.16.24} and  \eqref{09.16.22}, we obtain 
	\begin{equation*}
		-\im \int_0^T( \wt\vp, \pt_t\psi)_{L^2(\Om)}\df t+\int_0^T a(\wt\vp, \psi)\df t=\int_0^T (f,\psi)_{L^2(\Om)}\df t+(\im \wt\vp(0), \psi(0))_{L^2(\Om)}
	\end{equation*}
	for all $\psi\in C^1([0,T]; H_0^1(\Om;w))$ in \eqref{09.16.21} satisfying $\psi(T)=0$. Together with this and \eqref{09.16.25}, we deduce that 
	\begin{equation*}
		\wt\vp(0)=\vp^0. 
	\end{equation*}
	
	{\it Step 5}. Uniqueness. 
	
	Since the equation \eqref{09.16.1} is linear, we only need to show: if $\vp^0=f=0$, then $\vp=0$. Taking $v=\vp$ as a test function in the weak formulation, we obtain 
	\begin{equation}\label{09.16.26}
		\im \int_0^T\lg \pt_t\vp,  \vp\rg_{H^{-1}(\Om;w), H_0^1(\Om;w)}\df t+\int_0^Ta(\vp, \vp)\df t=0. 
	\end{equation}
	Taking conjugate of \eqref{09.16.26}, we get
	\begin{equation*}
		-\im \int_0^T \lg\ol{\pt_t\vp},\ol \vp\rg_{H^{-1}(\Om;w), H_0^1(\Om;w)}\df t+\int_0^T a(\ol\vp,\ol  \vp)\df t=0. 
	\end{equation*}
	Combining this and \eqref{09.16.26} and $a(\vp,\vp)=\int_\Om \nabla \vp \cdot A\nabla\ol\vp\df x=a(\ol\vp,\ol\vp)$, we deduce that 
	\begin{equation*}
		\int_0^t\f{\df}{\df t}\|\vp(t)\|_{L^2(\Om)}^2\df t=\|\vp(t)\|_{L^2(\Om)}^2, \mbox{ for all }t\in [0,T]. 
	\end{equation*}
	Hence $\vp=0$. 
	
	{\it Step 6}. From Steps 4 and 5, we get the existence and uniqueness of weak solution for \eqref{09.16.1}. From \eqref{09.16.20} we obtain \eqref{09.16.4}. We complete the proof of this theorem. 
\end{proof}

\begin{proof}[Proof of Theorem \ref{09.17.T1}]
	The proof is similar to that of Theorem \ref{09.16.T1}. 
	Indeed, the Galerkin approximation, the energy estimates, and the
	compactness argument can be performed in the same way on
	$\Omega_\delta$. Therefore, we only need to verify that the estimates
	are uniform with respect to $\delta$.
	
	Since $\Omega_\delta\s  \Omega$ and $\delta\in(0,\f{1}{8})$, the estimates
	used in the proof of Theorem \ref{09.16.T1}, including the Gronwall
	inequality, the estimate of $
	\|\partial_t\vp_\delta^m\|_{H^{-1}(\Omega_\delta;w_\delta)}$, 
	and the corresponding spectral estimate, remain valid on
	$\Omega_\delta$. Hence, passing to the limit as in the proof of
	Theorem \ref{09.16.T1}, we obtain a unique weak solution satisfying
	\begin{equation*} 
		\begin{split}
			&\esssup_{t\in[0,T]}
			\|\vp_\delta(t)\|_{H_0^1(\Omega_\delta;w_\delta)}
			+\|\partial_t\vp_\delta\|_
			{L^2(0,T;H^{-1}(\Omega_\delta;w_\delta))}
			\\
			&\leq C\left(
			\|\vp_\delta^0\|_{H_0^1(\Omega_\delta;w_\delta)}
			+\|f_\delta\|_{L^2(0,T;H_0^1(\Omega_\delta;w_\delta))}
			\right),
		\end{split}
	\end{equation*}
	where $C$ depends only on $\alpha$ and $T$, but is independent of
	$\delta$.
\end{proof}

\end{document}